\documentclass{amsart}
\usepackage{geometry}            
\usepackage{url} 		
\usepackage{graphicx}				
			
\usepackage{cleveref}					
\usepackage{amssymb}

\usepackage{amsmath}
\usepackage{amsthm}
\newtheorem{theorem}{Theorem}[section]

\newtheorem{lemma}[theorem]{Lemma}

\newtheorem{assumption}{Assumption}
\newtheorem{remark}{Remark}
\newcommand{\R}{\mathbb{R}}

\newcommand{\M}{\mathbb{M}}

\newcommand{\diff}{\mathrm{d}}
\newcommand{\N}{\mathbb{N}}
\newcommand{\MLS}{\mathcal{M}}
\newcommand{\MLAD}{\mathcal{T}}  

\newcommand{\Cret}{C_r}
\newcommand{\CRK}{C_{\mathrm{RK}}}

\newcommand{\lf}{\left}
\newcommand{\rt}{\right}
\newcommand{\weightkernel}{\omega}
\newcommand{\vv}{\mathbf{v}}
\newcommand{\euc}{\mathfrak{E}}

\newcommand{\Manoa}{M\=anoa}
\newcommand{\Hawaii}{Hawai`i }

\newcommand{\crcom}[1]{#1}
\newcommand{\gwcom}[1]{#1}
\newcommand{\thcom}[1]{#1}

\title[A Semi-Lagrangian scheme on embedded manifolds]{A Semi-Lagrangian scheme on embedded manifolds using generalized local polynomial reproductions}

\author{Thomas Hangelbroek} 
\address{Department of Mathematics, University of \Hawaii– \Manoa, 2565 McCarthy Mall,
Honolulu, HI 96822, USA}
\email{hangelbr@math.hawaii.edu}
\author{Christian Rieger} 
\address{Philipps-Universit\"at Marburg, Department of Mathematics and Computer Science,
Hans-Meerwein-Stra\ss{}e 6, 35032 Marburg, 
Germany}
\email{riegerc@mathematik.uni-marburg.de}
\author{Grady B. Wright} \address{Boise State University, 1910 University Drive, 83725, Boise, Idaho, USA}
\email{gradywright@boisestate.edu}
\thanks{Grady B. Wright  was supported by NSF grant DMS-2309712}

\date{\today}
\begin{document}
\begin{abstract}
We analyze rates of uniform convergence for a class of high-order semi-Lagrangian schemes for first-order, time-dependent partial differential equations on embedded submanifolds of $\R^d$ (including advection equations on surfaces) by extending the error analysis of Falcone and Ferretti \cite{FalconeFerretti}. A central requirement in our analysis is a remapping operator that achieves both high approximation orders and strong stability, a combination that is challenging to obtain and of independent interest. For this task, we propose a novel mesh-free remapping operator based on $\ell_1$ minimizing generalized polynomial reproduction, which uses only point values and requires no additional geometric information from the manifold (such as access to tangent spaces or curvature). Our framework also rigorously addresses the numerical solution of ordinary differential equations on manifolds via projection methods. We inclue numerical experiments that support the theoretical results and also suggest some new directions for future research.

 
\end{abstract}
\maketitle


%
%
%
\section{Introduction}
 In this article,
we treat  first order, time dependent PDEs
\begin{equation}
\label{eq:first_order}
\frac{\partial}{\partial t} v(x,t) = \vec{a}(x,t)\cdot \nabla v(x,t) + b(x,t) v(x,t) + c(x,t)
\qquad v(x,0) = v_0,
\end{equation}
with smooth coefficients and initial conditions
on certain closed, embedded manifolds of $\R^d$
by a  mesh-free {\em semi-Lagrangian} method.
This is a method, described below,
 which creates a discrete array $(\vv_{n,m})_{n\le N, m\le M}$ of values corresponding to samples of the true solution:
$$\vv_{n,m} \sim v(x_n, t_m)$$
where $x_n\in X$ is taken from a fixed  subset $X\subset \Omega$ having $N$ points.
Successive  vectors
$\vv_{\cdot,m}:=(\vv_{n,m})_{n\le N}$
 are based on solutions $\vv_{\cdot,m-1}$ at the previous time step by evolving the solution
along characteristic curves.  

Semi-Lagrangian (SL) methods have a long history in treating transport problems in various fields~\cite{FalconeFerrettiBook}, including numerical weather prediction~\cite{staniforth1991SL}, material science~\cite{nematic}, ecology~\cite{Plankton}, biophysics~\cite{farnudi2023dynamics}, and plasmas~\cite{sonnendrucker1999semi,Besse_Vlasov,Besse_HighOrder}.
A main focus in the development of these schemes 
has been to improve on finite-difference methods by providing stability
under more relaxed conditions.
This has been a key feature of SL methods especially in the regime of large Courant-Friedrichs-Lewy (CFL) numbers,
roughly $h_t/h_X$,
where
we use spatial resolution and (fixed) time steps
$$h_{X}:= \max_{x\in\Omega}\mathrm{dist}(x,X) \quad \text{and} \quad h_t:= t_m-t_{m-1}.$$ 
Large CFL numbers often arise in problems in numerical weather prediction,
and generally on transport problems on spheres, where it can be challenging to produce spherical grids amenable to high-order methods.
Here usually small mesh-sizes occur at the poles of latitude-longitude grids, or at element boundaries of geodesic grids, which can force the time-steps to be prohibitively small for explicit numerical methods.

The main idea behind SL methods is to use characteristic curves to
update the numerical solution:
 for a set of computational nodes, 
  integrate the ordinary differential equation governing the flow backwards in time for each node. 
  This leads to a new set of (dynamic) node positions, called {\em departure points}.
 
 A spatial approximation scheme is needed to transfer quantities between the computational nodes and the dynamical nodes; this is often called `remapping.'
When dealing with more complex underlying structures like triangulations, where mesh topology plays 
a 
role, remapping between structures can be quite difficult.
In the mesh-free scenario, remapping is substantially simpler.
It requires a semi-discrete approximation scheme 
$$\MLAD_{X}:\R^N\to C(\Omega).$$
which uses sampled values from the fixed set of nodes to construct an approximant
over the entire domain of interest: taking values on $X$ and producing a smooth function on $\Omega$.
This is used to approximate the solution at the 
departure point $(\xi_n,t_{m-1})$ of $(x_n,t_m)$:
$$v(\xi_n,t_{m-1})\sim 
\MLAD_{X}(\vv_{\cdot,{m-1}})(\xi_n).$$
Properties of this remapping operator crucially enter the convergence theory of SL methods, as considered in the seminal work~\cite{FalconeFerretti}.
\crcom{Due to stability properties of the remapping which are only known to be valid for piecewise linear approximation, the convergence analysis of \cite{FalconeFerretti} was limited to an error bound of the form $\mathcal{O}(h_X^2/h_t)$ which also could not benefit from higher smoothness of the solution $v$. } 
We aim at a convergence result of the form
\begin{equation*}
    |\vv_{n,m} - v(x_n,t_m)|\le C
\Bigl(h_t^{\ell}
+\frac{1}{h_t}\Bigl(\frac{h_{X}}{h_t}\Bigr)^k
\Bigr)
e^{(2\lambda+2)t_m}
\end{equation*}
where $h_t \le h_0$,
$h_{X} \lesssim h_t$ and $\ell,k$ denote the approximation orders the time integrator and the remapping operator, see \cref{thm:mainerror} for details.
\crcom{Theoretically this allows for  any convergence order $h^s$ for $s>0$ given smooth enough solutions $v$ and correspondingly high-order time integrators, see explanation after proof of \cref{thm:mainerror}.}

In order to achieve this error bound, we need a remapping operator with two properties: high approximation orders
and strong stability. 
High approximation orders are needed in order to achieve high convergence rates for the SL scheme 
given that the solution is regular enough. There are many well known schemes achieving this property.

However, the second key element is a very restrictive stability property of the scheme. 
This stability is often expressed in some variants of the Lebesgue constant for polynomial interpolation. 
In \cite{FalconeFerretti} the approximation operator is assumed to be interpolatory $\MLAD_{X}\vec{c}(x_n) = \vec{c}_n$ for $x_n\in X$
and weakly contractive:
$\sup_{x\in \Omega} |\MLAD_{X} \vec{c}(x) |\le \|\vec{c}\|_{\ell_{\infty}(X)}.$
Other than piecewise linear
interpolation, there seem to be no examples of such operators.
Thus  the  methods and associated error analyses considered in \cite{FalconeFerretti} are, in practice, limited  to being low (second) 
order in space.
%
%
%
\subsubsection*{Novel remapping operator}
Our approach departs from the rather strict assumptions of \cite{FalconeFerretti} by dropping the interpolation assumption, and by permitting a weaker stability
requirement: $$\sup_{x\in \Omega} |\MLAD_{X} \vec{c}(x) |\le (1+\epsilon) \|\vec{c}\|_{\ell_{\infty}(X)},$$ for user defined $\epsilon>0$.
Precise requirements for $\MLAD_{X}$ are given in \Cref{ass:A_approximator}.
We show that the recent result \cite{HRW_LPR} on local polynomial reproductions
guarantees the existence of such spatial approximators with arbitrarily high approximation orders, 
and thus deliver suitable convergence rates for the semi-Lagrangian 
scheme (comparable to those in \cite{FalconeFerretti}).

In particular, we point out that our remapping operator in the setting of the error bound \eqref{eq:motivation} requires the solution of a linear program with $n_k\sim \theta^{d_{\M}}_k h^{-\frac{d_{\M}}{\sigma}}_{X}$ many degrees of freedom, see \Cref{sec:ell1} for details.

We note that similar mesh-free approaches on manifolds have been attempted recently in \cite{PLR, SW} using  kernel  interpolation to perform the spatial approximation.
Although  kernel interpolation enjoys high orders of convergence, 
it does so without guarantee of stability
(indeed, only in rare instances are Lebesgue constants known to be controlled independent of the problem size \cite{HNW-poly}). 
Although both of these articles show compelling results, neither is able to use the restrictive convergence theory framework of \cite{FalconeFerretti}.
%
%
%
\subsection{Main contributions}
In this article  we focus on  first order advection equations over a class of boundary-free, 
embedded submanifolds of $\R^d$ as SL methods are often applied in surface advection related problems.
Thus, the computational domain $\Omega= \M\subset \R^d$ is a smooth, $d_{\M}$-dimensional closed manifold. 

We propose a novel remapping operator based on a norming set approach and resulting local polynomial reproductions. This approach is closely related to moving least squares (MLS) techniques but exhibits some key differences. 
Our approach considers the restriction of polynomials from the ambient space to the embedded manifold. Algorithmically, we do not even need to know the dimension of the linear space of the restricted polynomials in advance. This is in particular attractive in the case of unknown, 
complicated or moving surfaces.  Moreover, we do not require any approximation of the tangent space of the embedded surface. 
(There is an accompanying MLS method that we consider in our numerical experiments which similarly does not require 
local approximation of the manifold, in contrast to conventional MLS methods as in  \cite{GrossTrask,SoberMLS,YanGMLS}.) 
\thcom{This results in an algorithm enjoying  the aforementioned convergence result
which only requires
 $h_{X} \le C h_t$. In particular, it does not require a CFL condition, which is typically of the form $h_t \le c h_{X}$.
 For sufficiently smooth solutions,
 the condition $h_X \sim h_t^{\sigma}$, with $\sigma> 1$, permits
 high-order convergence in terms of $h_X$.}

Moreover, we present a feasible numerical procedure in order to implement the novel remapping operator and hence also the SL method. Here we also place a focus on explicit Runge-Kutta methods on manifolds using a projection approach.
%
%
%
\subsection{Outline}
In the next section, we discuss the linear advection equation, along with
an error analysis for a SL approach. \Cref{sec:ell1} treats the remapping operator and discusses its practical implementation.
\Cref{sec:ODE} presents a  high precision ODE solver for embedded manifolds.
Solving
systems of ordinary differential equations on manifolds is a topic which has 
received much attention. To some extent, the particularities of the ODE solver 
are not so important
for our main results,
although
achieving high precision and producing smooth solutions
is critical.
 We include this section for the sake of completeness, 
 to demonstrate that hypotheses on our time integrators are rigorously satisfied,
and to describe how our solution was implemented.
We conclude with \Cref{sec:NumericalResults}
which gives numerical results for our algorithm, and further discusses the stability
assumption on the operator $\MLAD_{X}$.
%
%
%
\section{A semi-Lagrangian scheme for advection problems} 
\label{sec:Theory}
Let  $\M\subset \R^d$ be a smooth manifold
with  Riemannian metric inherited from $\R^d$, so that the gradient $\nabla u(x)\in T_x\M$ of a function $u$ at $x$ is defined
in the usual way: 
for every tangent vector $\vec{a}\in T_x\M$,  $\vec{a}\cdot \nabla u(x) =  du(x)\, \vec{a} $.
Equivalently, if $\tilde{u}:\R^d\to \R$ is a $C^1$  extension of $u$ (so that $\tilde{u}|_{\M} = u$), then for $x\in\M$,
$\nabla u (x)= P \nabla_{\R^d} \tilde{u}(x)$, where $P: T_x\R^d\to T_x\M$ is the orthogonal projector
of $ T_x\R^d = T_x\M\oplus N_x \M$
onto $T_x\M$.
%
%
%
\subsection{Solution via characteristics}
We consider the PDE \cref{eq:first_order}, with $\vec{a}:\M\times \R\to T\M$ a $C^k$ vector field,
and scalar functions $b,c,v_0\in C^k(\M)$.
Solutions can be found via the method of  characteristics:
for  $t_0\in \R$ and $x_0\in\M$,
 the  (reverse time) characteristic   $y:(-\epsilon,\epsilon)\to \M$  solves
\begin{equation}
\dot{y}(s) = \vec{a}(y(s),t_0-s),\qquad y(0) =x_0. \label{eq:ODE}
\end{equation}
For any function $v:\M\times (t_0-\epsilon,t_0+\epsilon) \to \R$,
the   function 
${v}_{\flat}:s\mapsto v(y(s),t_0-s)$ 
satisfies
$\dot{v}_{\flat}(s) 
= \vec{a}(y(s) ,t_0-s)\cdot \nabla v(z(s),t_0-s) - v_t(y(s),t_0-s).
$
Thus $v$ is a solution to \cref{eq:first_order} if and only if ${v}_{\flat}$ 
satisfies
$\dot{v}_{\flat}(s)
=  -b(y(s),t_0-s){v}_{\flat}(s)- c(y(s),t_0-s)$
for any $x_0,t_0$. 
We solve for $v$ by  computing the integrals
\begin{align} 
\mu_{x_0,t_0}(s) &:=\exp\bigl(\int_0^s b(y (\tau), t_0-\tau)\diff \tau),
\label{eq:mu}\\ 
\nu_{x_0,t_0}(s) &:=\int_0^{s} \mu(\tau) c(y(\tau),t_0-\tau)\diff \tau.
\label{eq:nu}
\end{align}
This provides the solution
\begin{equation}
v(x_0,t_0)  =\mu_{x_0,t_0}(s) v(y(s),t_0-s)  +\nu_{x_0,t_0}(s) .
\label{eq:Soln}
\end{equation}
The  {\em evolution map} $F$, as described in \cite[Chapter 4]{marsden}, is defined for $s,t_0\in\R$ and $x_0\in \M$
as $F_{s,t_0}(x_0) = y(s)$. Thus $F_{t_0,t_0}(x_0)=x_0$ and 
$\frac{\partial }{\partial s} F_{s,t_0}(x_0) 
=  \vec{a}(F_{s,t_0}(x_0),s).$
In addition $F\in C^k$  when $\vec{a}\in C^k$.
This fact, in conjunction with \cref{eq:Soln},
 guarantees
that $v$ is smooth provided $\vec{a}$, $b$, $c$ and $v_0$ are smooth, see also \cite[Appendix B]{Duistermaat2000}. 

A further fact is that the growth of $\mu$, $\nu$ and $v$  can be controlled by using
constants  $\|v_0\|_{\infty}$, 
 $\overline{\lambda} = \sup b(x,t)$, and $\|c\|_{\infty}$.
For any $0<s\le t_0$, we have $|\mu_{x_0,t_0}(s)|\le e^{\overline{\lambda} s}$ 
and thus $|\nu_{x_0,t_0}(s)|\le  \|c\|_{\infty}  s e^{\overline{\lambda} s}$.
Consequently, 
$$|v(x_0,t_0)| \le (\|v_0\|_{\infty}+t_0\|c\|_{\infty})e^{\overline{\lambda} t_0}.$$
%
%
%
\subsection{Algorithm}
To compute the solution to \cref{eq:first_order} numerically, one needs to solve the ODE \cref{eq:ODE},
while the integrals
for $\mu(s;x_0,t_0)$ and 
$\nu(s;x_0,t_0)$
can be approximated by quadrature,
using approximate values sampled from the characteristic.

We'll assume there are numerical integrators for \cref{eq:ODE}. At first, consider a one-step scheme
\thcom{$$
y(h_t)
\sim
\tilde{y}(h_t):=
\Phi_{(x_0,t_0)}(h_t)
$$
satisfying
 $\lim_{h_t\to 0}\Phi_{(x_0,t_0)}(h_t) -x-h_t\vec{a}(x,t-h)=0$,
although for our error analysis, we'll assume (the  stronger)  
 \cref{ass:A_integrator}, given below.
We assume the integrals in \cref{eq:mu} and \cref{eq:nu} can be computed by quadrature:
$Q^{(1)}_{x_0,t_0}(s) \sim \mu_{x_0,t_0}(s)$ and $Q^{(2)}_{x_0,t_0}(s) \sim \nu_{x_0,t_0}(s)$
(doing this with high accuracy, as required by  \cref{ass:A_integrator},  is a relatively straightforward).}

We then have a computed solution at time $t_m$:
\begin{equation*}
v(x,t_m) \sim 
\thcom{Q^{(1)}_{x,t_m}(h_t) \ v\bigl( \Phi_{(x,t_m)}(h_t),t_{m-1}\bigr)+Q^{(2)}_{x,t_m}(h_t).}
\end{equation*} 
By replacing the value
$v(\xi_n,t_{m-1})$  at the  departure point $\xi_n:=\Phi_{(x,t_m)}(h_t)$
with the help of the remapping operator $\MLAD_{X}:\R^{N}\to C(\Omega)$,
this yields a space/time discrete solution 
\begin{equation}
\label{eq:fully_discrete}
\vv_{n,m} = \thcom{ 
Q^{(1)}_{x_n,t_m}(h_t) \Bigl(\MLAD_{X} (\vv_{\cdot,m-1})\Bigr)(\xi_n)
+Q^{(2)}_{x_n,t_m}(h_t)}
\end{equation}
at time $t_m$ and position $x_n \in X$.
%
%
%
\subsection{Key assumptions and main theorem\label{sec:main_theorem}}
In addition to the smoothness assumptions $\vec{a},b,c\in C^k$, we make assumptions on the time integrator and remapping operator described below. These are explored in \Cref{sec:ell1} (for the remapping operator) and \Cref{sec:ODE} (for the time integrator).

We assume the time integrators have consistency order $\ell$ in the following  sense.
\begin{assumption}
\label{ass:A_integrator}
\thcom{For  $\ell\in \N$ there is a  time integrator  $\Phi$ for (\cref{eq:ODE})   with consistency order $\ell$.
That is, there is a constant $C_{\Phi}$ so that  for $(x_0,t_0) \in \M\times(0,\infty)$
and $h_t>0$,
\begin{equation}
\label{eq:One_step_error}
 |\Phi_{(x_0,t_0)}(h_t)
 - y(h_t)|\le C_{\Phi} h_t^{\ell+1}
\end{equation}
holds. }
Furthermore, we assume that the quadrature rules for $\mu$ and $\nu$ are sufficiently
precise that for all $x_0\in \M$ and $t_0\le T$,
\begin{equation}
\label{eq:quadrature}
|Q^{(1)}_{x_0,t_0}(h_t) - \mu_{x_0,t_0}(h_t)| \le  C_{Q} h_t^{\ell+1}
\quad
\text{ and }\quad
|Q^{(2)}_{x_0,t_0}(h_t) - \nu_{x_0,t_0}(h_t)|\le  C_{Q} h_t^{\ell+1}.
\end{equation}
\end{assumption}
\begin{remark}
\label{quad}
\normalfont
Note that \cref{eq:quadrature} follows almost directly once one has a suitably high performing one step method $\Phi$, since
the integrals  involved in \cref{eq:mu} and \cref{eq:nu} are univariate.
Indeed, this can be accomplished with a 
simple quadrature rule over $[0,1]$ (e.g., Newton–Cotes if $\ell$ is not too large, or Gaussian Quadrature).
For weights $w_k$ and nodes $t_k\in [0,1]$ $k=0,\dots M$, suppose $ |\int_0^{1} f(\tau)\diff \tau - \sum_{k=0}^{M} w_k f(t_k) |\le C_{\text{quad}} \max_{z\in[0,1]} |f^{(\ell)}(z)| $. Then for $G\in C^{\ell}(\M)$, \cref{eq:One_step_error} along with a  straightforward scaling argument guarantees that 
\begin{eqnarray*}
\Bigl| \int_0^{h_t} G(y(\tau))\diff \tau - h_t \sum_{k=0}^M w_k G(\tilde{y}(t_k h_t)) \Bigr|
& =& 
h_t \Bigl|\int_0^{1} G(y(\tau h_t))\diff \tau - \sum_{k=0}^M w_k G(\tilde{y}(t_k h_t)) \Bigr|\\
&\le&
C h^{\ell+1} 
\end{eqnarray*}
with $C = C_{\text{quad}} \|G\circ y\|_{C^{\ell}[0,h_t]} +C_{\Phi} \|G\|_{Lip} \left(\sum_{k=0}^M | w_k |  \right) $.
\end{remark}

We assume that $\MLAD_{X}$ satisfies the following approximation and stability conditions.
\begin{assumption}\label{ass:A_approximator}
For every $k\in \N$ there exist constants $\Gamma>0$ and $C_T$ so that 
for every $\epsilon>0$ and any $X\subset \Omega$ with $h_{X} \le \Gamma \epsilon$,
there is an approximation operator $\MLAD_{X}:\R^{N}\to C(\Omega)$ which satisfies the following:
\begin{enumerate} 
\item Stability: 
$$\|\MLAD_{X} \vec{c}\|_{C(\Omega)} \le (1+\epsilon)\|\vec{c}\|_{\ell_{\infty}(X)}$$
\item Approximation: 
$$\|
(\MLAD_{X} \circ S_{X}) f - f\|_{C(\Omega)} 
\le 
C_T \Bigl(\frac{h_{X}}{\epsilon}\Bigr)^{\thcom{k+1}} \|f\|_{C^{\thcom{k+1}}(\Omega)}.$$
Here $S_{X}:C(\Omega)\to \R^{N}$ is the sampling operator: $S_{X} f = f|_{X}$.
\end{enumerate}
\end{assumption}
Based on these assumptions and recalling $\overline{\lambda}= \sup b(x,t)$, we now present error estimates for the SL method.
%
%
\begin{theorem}\label{thm:mainerror}
If \Cref{ass:A_integrator}
and \Cref{ass:A_approximator} hold then
there exists a positive constant $C$ so that if 
$h_t \le h_0$ where $h_0=\min(1/\overline{\lambda}, {C_{Q}}^{-1/\ell})$
and
$h_{X} \le \frac{\Gamma}{3} h_t$, the semi-Lagrangian method satisfies
$$|\vv_{n,m} - v(x_n,t_m)|
\le 
C
\Bigl(h_t^{\ell}
+\frac{1}{h_t}\Bigl(\frac{h_{X}}{h_t}\Bigr)^{\thcom{k+1}}
\Bigr)
e^{(2\lambda+2)t_m}.
$$
\end{theorem}
\begin{proof}
By applying definitions, we  estimate $\mathcal{E}_m := \|\vv_{\cdot,m} - S_{X} v(\cdot,t_m)\|_{\ell_{\infty}(X)}$ 
by
\begin{eqnarray*}
\mathcal{E}_m \le\max_{x_n\in X}
\Bigl| Q^{(1)}_{x_n,t_m}(h_t) \Bigl( \MLAD_{X} \vv_{\cdot,m-1}\Bigr)(\xi_n)
- \mu_{x_n,t_m}(h_t) v\bigl(y(h_t), t_{m-1}\bigr) \Bigr|
+C_Q h_t^{\ell+1}
\end{eqnarray*}
since the part due to $\nu$
 can be controlled 
 simply by using \Cref{ass:A_integrator}.
 By adding and subtracting
$ Q^{(1)}_{x_n,t_m}(h_t) v\bigl(y(h_t), t_m-h_t\bigr) $, 
we have
$$\mathcal{E}_m
\le \max_{x_n\in X}
\Bigl| Q^{(1)}_{x_n,t_m}(h_t)\Bigl[ \Bigl( \MLAD_{X} \vv_{\cdot,m-1}\Bigr)(\xi_n)
-  v\bigl(y(h_t), t_{m-1}\bigr) \Bigr] \Bigr|
+
C h_t^{\ell+1} \max(1, t_m)e^{\overline{\lambda} t_m}
$$
with $C = C_Q(1+\|v_0\|_{\infty}+\|c\|_{\infty})$.
 
To estimate
$|Q^{(1)}_{x_n,t_m}(h_t)| \left[\Bigl(\MLAD_{X} \vv_{\cdot,m-1}\Bigr)(\xi_n) -v\bigl(y(h_t),t_{m-1}\bigr)\right]$,
we note that the bound on the growth of $\mu$ guarantees
$$|Q^{(1)}_{x_n,t_m}(h_t)|\le e^{\overline{\lambda}h_t}+ C_Q h_t^{\ell+1}
\le  e^{\overline{\lambda}h_t}+ h_t
.$$
To estimate the bracketed expression we introduce
$\MLAD_{X} \bigl( S_{X} v(\cdot ,t_{m-1})\bigr)$,
and 
$v(\xi_n, t_{m-1})$, 
obtaining
\begin{eqnarray*}
\left| \Bigl(\MLAD_{X} \vv_{\cdot,m-1}\Bigr)(\xi_n) -v(y(h_t),t_{m-1})\right|
&\le&
\left\| \MLAD_{X} \vv_{\cdot,m-1} - \MLAD_{X} \bigl( S_X v(\cdot ,t_{m-1})\bigr)\right\|_{\infty}\\
&&+ \left\|\MLAD_{X} \bigl( S_X v(\cdot ,t_{m-1}) \bigr)- v(\cdot, t_{m-1}) \right\|_{\infty}\\
&&+|v(\xi_n, t_{m-1}) -v\bigl(y(h_t),t_{m-1}\bigr)|.
\end{eqnarray*}
The first line can be estimated by the stability of $\MLAD_{X}$: 
$$\left\| \MLAD_{X} \vv_{\cdot,m-1} - \MLAD_{X} \bigl( S_X v(\cdot,t_{m-1})\bigr)\right\|_{\infty} 
\le
(1+\epsilon)\mathcal{E}_{m-1}.$$
The second line can be estimated with the approximation property:
$$\left\|\MLAD_{X} \bigl( S_X v(\cdot ,t_{m-1}) \bigr)- v(\cdot, t_{m-1}) \right\|_{\infty}
\le 
C_T\Bigl(h_X/\epsilon\Bigr)^{\thcom{k+1}} \|v(\cdot,t_{m-1})\|_{C^{\thcom{k+1}}(\Omega)}.
$$
The third line is simply the Lipschitz property of $v$ and the  error due to the time integrator:
$$|v(\xi_n, t_{m-1}) -v(y(h_t),t_{m-1})| \le \|v\|_{Lip} |\xi_n - y(h_t)| \le 
 C_{\Phi} \|v\|_{Lip}  h_t^{\ell+1}.$$
Thus  the recursive estimate
$\mathcal{E}_{m} 
\le 
C (h_t^{\ell+1}   +({h_{X}}/{\epsilon})^{\thcom{k+1}} )+ 
(e^{\overline{\lambda }h_t}+h_t) (1+\epsilon) \mathcal{E}_{m-1}
$ 
holds
after combining constants, namely by letting $C :=  C_{T}\thcom{\|v\|_{C^{k+1}(\Omega)}}+ C_{\Phi}\|v\|_{Lip} $.

Define $\rho(h_t,h_{X}) := C \Bigl(h_t^{\ell+1} + \bigl(h_{X}/\epsilon)^{\thcom{k+1}}\Bigr)$
and $L:= (e^{\overline{\lambda} h_t} +h_t)(1+\epsilon) $.
Then
$$\mathcal{E}_{m}\le \rho(h_t,h_{X}) + L\mathcal{E}_{m-1}
\le 
\rho(h_t,h_{X}) (1+L)+ L^2\mathcal{E}_{m-2}
\le \frac{L^m-1}{L-1} \rho(h_t,h_{X}).$$
Since 
 $1+ (\overline{\lambda}+1) h_t\le e^{\overline{\lambda} h_t}+h_t \le 1+ (2\overline{\lambda}+1) h_t$
 and  $\epsilon\le \frac{1}{3} h_t$,
 the quantity
 $L$ satisfies
 $1+(\overline{\lambda}+1) h_t \le L\le 1+ (2\overline{\lambda}+ 2) h_t$,
so 
$$\mathcal{E}_{m}\le 
\frac{\rho(h_t,h_{X})}{(\overline{\lambda}+1) h_t}  \Bigl(1+ \frac{(2\overline{\lambda}+ 2)}{m}t_m\Bigr)^m\le 
\frac{\rho(h_t,h_{X})}{(\overline{\lambda}+1) h_t} e^{(2\overline{\lambda}+2)t_m}$$
and the result follows.
\end{proof}
\crcom{In order to achieve an overall approximation rate $s>0$ for problems with enough smoothness, we can fix an $\sigma>1$ and pick a time integrator such that $\ell=s \sigma$ holds. The time step should be chosen as $h_X \lesssim h_t \sim h^{1/\sigma}_X$. Finally, we choose 
\thcom{$k=\lceil\frac{s\sigma+2-\sigma}{\sigma-1}\rceil=\frac{s\sigma+2-\sigma}{\sigma-1}+r$}
with $r\in [0,1)$. 
These choices yield 
\begin{equation}\label{eq:motivation}
    |\vv_{n,m} - v(x_n,t_m)|\le C
\Bigl(h_X^{\frac{s\sigma}{\sigma}}
+h^{-\frac{1}{\sigma}}_{X}\Bigl(h^{\frac{\sigma-1}{\sigma}}_{X}\Bigr)^{\thcom{\frac{s\sigma+2-\sigma}{\sigma-1}+r+1}}
\Bigr)
e^{(2\lambda+2)t_m}\le 
C h^s_X e^{(2\lambda+2)t_m},
\end{equation}
using that $r\frac{\sigma-1}{\sigma}>0$.
The choice of $k$ determines the degree of polynomial exactness of the remapping operator $\MLAD_{X}$ and hence its numerical costs. 
}
%
%
%
\section{The remapping operator}
\label{sec:ell1}
In this section, we discuss the existence and practical implementation of the remapping operator $\MLAD_{X}$ satisfying
\Cref{ass:A_approximator}.
Our approach is based on local polynomial reproduction (LPR) given in \cite{HRW_LPR}: for a set $\Omega\subset \R^N$
 a polynomial degree $k$ and a point set $X\subset \Omega$, an LPR of degree $k$ 
 is a family of smooth functions $a(x_j,\cdot):\Omega \to \R$ 
 so that 
 \begin{itemize}
 \item $\sup_{z\in\Omega}\sum_{j=1}^N |a(x_j,z)|<\infty$  is bounded,
\item  for each $z\in\Omega$, the support of $a(\cdot,z)$ (as a subset of $X$) is contained in a neighborhood of $z$, and
  \item for all $p\in \mathcal{P}_k(\R^d)$, 
 $\sum_{j=1}^N a(x_j,z) p(x_j) = p(z)$, where $\mathcal{P}_k(\R^d)$
 denotes the space of polynomials of degree $\leq k$.
\end{itemize}
If $\M$ is an algebraic manifold,
\cite[Proposition 6.1]{HRW_LPR} guarantees for any polynomial degree $k$,
that there exist constants $\Gamma_{k}>0$ and $\varTheta_{k}$ so that
if $\epsilon>0$ that if  $X\subset \M$ has $h_{X}< \Gamma_{k} \epsilon$ 
there exists an LPR $a:X\times \M\to \R$ with the property that 
$\sum_{j=1}^N |a(x_j,z)|\le 1+\epsilon$ 
and $\mathrm{supp}\bigl(a(\cdot,z) \bigr)\subset B(z, \varTheta_k\frac{h_X}{ \epsilon})$. 
In fact, this result holds for suitable subsets $\Omega\subset \M$, in which case the
constants $\Omega_k$ and $\varTheta_k$ depend on cone parameters of $\Omega$ --
by considering unions of Lipschitz regions with similar constants $\Gamma_k,\Theta_k$,
 this result holds on certain piecewise algebraic manifolds.
 In short, although it is unclear if the result extends to arbitrary embedded Riemannian manifolds,
 it does hold on a robust class of subsets (which includes, but is much larger than, compact algebraic
 manifolds).
\subsubsection*{Moving least squares}
A  practical application of the LPR to produce a remapping operator which is easily computed, although not aligned with the theory of \Cref{sec:Theory},
is the moving least squares operator:
\begin{align}
\label{eq:Pol_MLS}
 \mathcal{M}_{X}f(z) = 
    p_z^*(z),\text{ where } p_z^*:=\arg\min_{p \in \mathcal{P}_k} 
    \sum_{j=1}^N \left(p(x_j) -f(x_j) \right)^{2} \weightkernel(x_j,z).
\end{align}
as considered in \cite[Sections 5, 6.2]{HRW_LPR}. Here $\weightkernel$ is a weight, which normally takes the form $\weightkernel(x,z)  = \phi(\frac{x-z}{h_X})$
for some radially symmetric function $\phi$ which is  supported  in a ball $B(0,r_1)$ and  bounded away from zero on a ball $B(0,r_2)\subset B(0,r_1)$.
It is well known that $ \mathcal{M}_{X}$ has the form 
$ \mathcal{M}_{X}f = \sum f(x_j) b(x_j,\cdot)$ 
for an LPR $b:X\times \M\to \R$ of degree $k$.

If $X$ is {\em quasi-uniform}, meaning that $q_X =\min |x_j-x_k|$ is on par with $h_X$, 
so $h_X/q_X$ is controlled by some fixed constant $\rho$,
then $\sum |b(x_j,z)| \le C \rho^{\frac{d_{\M}}{2}}$ by \cite[Proposition 5.1]{HRW_LPR}.
Furthermore, by \cite[Proposition 6.3]{HRW_LPR},
we have 
$$|f(z) -  \mathcal{M}_{X}f(z)| \le C h^{\thcom{k+1}} \|f\|_{C^{\thcom{k+1}}(\M)}.$$
Thus $ \mathcal{M}_{X}$ satisfies item 2 (with $\epsilon \sim 1$)
 but not item 1 of \Cref{ass:A_approximator}; it provides high-order approximation rates, but does not appear
 to be weakly contractive.
 Nevertheless, this is a reasonable candidate for a remapping operator and we use it for comparison purposes in the numerical experiments of \Cref{sec:NumericalResults}.
%
%
%
\subsection{Remapping by $\ell_1$minimization\label{sec:ell1_remap}}
We can calculate a piecewise smooth, globally Lipschitz local polynomial reproduction by
selecting for every $z\in \Omega$, a vector $a^*:X\times \Omega\to \R$, which is
constructed as follows: 
for
$x_j \notin B(z, \varTheta_k\frac{h_X}{\epsilon} )$ we set $a^*(x_j,z)=0$, while 
in $B(z, \varTheta_k\frac{h_X}{\epsilon} )$, we have
 \begin{equation}\label{eq:shape_function}
 a^*(\cdot,z)\left|_{X\cap B(z, \varTheta_k\frac{h_X}{\epsilon}  )}\right. :=
 \arg\min \sum_{x_j \in X \cap B(z, \varTheta_k\frac{h_X}{\epsilon}  )} |a_{j}|,\quad\text{ subject to } V_z a= p_z 
 \end{equation}
 where $V_z$ is the Vandermonde matrix for a basis $(p_1,\dots,p_M)$ of $\mathcal{P}_k$ using local center set 
 $X\cap B(z, \varTheta_k \frac{h_X}{\epsilon})$,
 and $p_z$ is the vector $\bigl(p_j(z)\bigr)_{j\le M}$.
 We then set 
\begin{equation}\label{eq:approximator}\MLAD_{X} \vec{c}(z) := \sum_{j=1}^N a^*(x_j,z) c(x_j).
 \end{equation}
 Note that, in order to use this in the 
 semi-Lagrangian scheme, we need to 
 evaluate it at the departure points 
 $\xi_n=x_n +h_t\Phi_{\vec{f}}(x_n,h_t)$, for each
 $x_n\in X$.
 
We point out that \cref{eq:shape_function} formally resembles the classical \emph{basis pursuit} method. To this end, we consider the polynomials $\mathcal{P}_k$ as dictionary of functions. The Vandermonde matrix $V_z$ is the measurement matrix and $p_z$ are the measurements for $z$ arbitrary but fixed.
 This means that we assume to have data $p_z=(p_1(z),\dots,p_{M}(z))^{\top}=(f(1),\dots,f(M))^{\top}$.
 As $\#\bigl(X\cap B(z, \varTheta_k \frac{h_X}{\epsilon}  ) \bigr) \gg M$, we have potentially several vectors $a(\cdot,z) \in \R^{n}$ such that
 $f(m) =p_{m}(z)= \sum_{\xi } a(\xi,z) p_{m}(\xi) $ holds for all $1\le m \le M$. 
 
 In this case the index is the variable in the basis pursuit setting, which asks for the $\ell_1$-norm minimal solution vector. 
The solution $a^{*}(\cdot,z)$ is then extended by zero to all of $X$. This indicates the sparsity of the solution vector.
 \begin{lemma}
 \label{lem:lemma_assumption2}
 If $\Omega$ is a  compact region  
 of an algebraic
 manifold $\M$
 which satisfies a cone condition, 
 the approximator $\MLAD_{X}$ given in \cref{eq:approximator} satisfies \Cref{ass:A_approximator}.
 \end{lemma}
   We note that if $\Omega=\bigcup_{j=1}^K \Omega_j$ is a union of compact regions $\Omega_j$ of algebraic
 manifolds $\M_j$, each satisfying a cone condition,
 then 
 \thcom{$\MLAD_{X}$}
 satisfies \Cref{ass:A_approximator}
 \thcom{on such $\Omega$}
 as well.
 \thcom{In this case, the constants $\Gamma = \min_{j\le K} \Gamma_j$ and $C_T= \max_{j\le K} C_{T,j}$ are determined by the minimum and maximum of the constants for each of the pieces $\Omega_j$ (each constant $C_{T,j}$ and $\Gamma_j$ is  guaranteed by \Cref{lem:lemma_assumption2}).
 The precise dependence of these constants on the geometric quantities can be found in \cite[Section 3.2]{HRW_LPR}.  In practice, we do not assume to have access to such constants, see \Cref{sec:lebesgue} for a discussion.}
 \begin{proof}
 It follows from \cite[Proposition 6.1]{HRW_LPR}
  that for any $z$, \cref{eq:shape_function}
  is well posed and the solution satisfies 
  $\sum_{j=1}^N |a^*(x_j,z)|\le 1+\epsilon$.
  Consequently, we have the stability bound $\|\MLAD_{X}\|_{\ell_{\infty}(X)\to C(\Omega)}\le 1+\epsilon$.
  This is item 1 of \Cref{ass:A_approximator}.
 
By the stability, locality and polynomial reproduction properties of $a^*$, we then have the standard estimate
\begin{align*}
  \left| f(z) -\MLAD_{X}\thcom{(f|_X)(z)}\right| 
  &\le 
    \left| f(z) -p(z)\right|+\left|\MLAD_{X} \thcom{(p|_X)}(z) - \MLAD_{X}\thcom{(f|_X)}(z)\right|\\
  &\le   (2+\epsilon) 
    \|f-p\|_{L_{\infty}(\Omega\cap B(z,\varTheta_k {h_X}/{\epsilon} ))}
\end{align*}
Because the expression on the left is independent of $p$, we can replace this by the infimum:
$$ \left| f(z) -\MLAD_{X}\thcom{(f|_X)}(z)\right| \le (2+\epsilon) \inf_{p\in \mathcal{P}_k} \|f-p\|_{L_{\infty}(\Omega\cap B(z,\varTheta_k {h_X}/{\epsilon} ))}.$$
If $f\in C^{\thcom{k+1}}(\Omega)$, then we can take an extension $\mathcal{E}f\in C^{\thcom{k+1}}(W)$ defined on a neighborhood $W$ of $\Omega$ in \gwcom{$\R^d$} 
 with $\|\mathcal{E}f\|_{C^{\thcom{k+1}}(W)}\le C\|f\|_{C^{\thcom{k+1}}(\Omega)}$.
Thus 
\begin{align*}
&\inf \|f-p\|_{L_{\infty}(\Omega\cap B(z,\varTheta_k{h_X}/{\epsilon} ))}
\le
\inf \|\mathcal{E}f-p\|_{L_{\infty}(W\cap B_{\euc}(z,\varTheta_k{h_X}/{\epsilon}))} \\
&\le
C 
\Bigl(
\frac{h_X }{\epsilon}
\Bigr)^{\thcom{k+1}}
\|\mathcal{E}f\|_{C^{\thcom{k+1}}(W)} 
\le
C 
\Bigl(
\frac{h_X }{\epsilon}
\Bigr)^{\thcom{k+1}}
\|f\|_{C^{\thcom{k+1}}(\Omega)}.
\end{align*}
This is item 2 of \Cref{ass:A_approximator}.
\end{proof}
%
%
%
\subsection{Practical implementation of $\ell_1$minimization}
This section provides a well-known reformulation of the $\ell_1$ minimization as linear program, see e.g. \cite[comment below Theorem 3.1]{Foucart:Rauhut:2013}.
We sketch the idea for the reader's convenience.
We point out that there are several algorithms to solve this convex optimization problem, see e.g. \cite[Section 15]{Foucart:Rauhut:2013} \crcom{or \cite{Wrightbook}}.
Let $\vec{A} \in \R^{m \times n}$, $\vec{b}\in \R^{m}$. Define an admissible set
$	\mathcal{A}=\left\{ \vec{x} \in \R^{n} \ : \ \vec{A}\vec{x}=\vec{b}  \right\}\subset \R^{n}$ and consider the optimization problem
$\vec{x}^{\star}=\arg\min_{\vec{x}\in \mathcal{A}} \| \vec{x}\|_{\ell_1}$.
The idea is to consider a componentwise splitting
$$x_j^+ := \begin{cases} x_j&x_j\ge0\\0&x_j<0,\end{cases}\qquad
x_j^{-}:=  x_j^{+}-x_j.$$
Next, consider the optimization problem, 
\begin{equation}\label{eq:lp}
	\begin{aligned}
		&\text{minimize} 
        \sum_{j=1}^{n} (x^{+}_{j} + x^{-}_{j}) \\
		&\text{subject to } {(x^+,x^-)\in \R^{2n}} \ \text{such that }
        \vec{A} ({\vec{x}\: }^{+}-{\vec{x }\: }^{-})=\vec{b}, \quad {\vec{x}\: }^{+},{\vec{x}\:}^{-}\ge 0. 
	\end{aligned}
\end{equation}
This can be solved with a standard linear program (e.g., \texttt{linprog} in \textsc{Matlab}).

\begin{lemma}
The vector
$(x^+,x^-)$ solves \cref{eq:lp} if and only if  
$x^+$ and $ x^-$ are the positive and negative parts of $\arg\min_{x\in \mathcal{A}} \|x\|_{\ell_1}$
(so $x^+- x^-=
\arg\min_{x\in \mathcal{A}} \|x\|_{\ell_1}$).
\end{lemma}
\begin{proof}
If $(x^+,x^-)$ solves \cref{eq:lp}, then  
$x^+-x^-\in \mathcal{A}$. So 
if $x^* :=\arg\min_{x\in \mathcal{A}} \|x\|_{\ell_1}$, then
$\|x^*\|_{\ell_1} \le x^++x^-$.

To show that
the splitting into positive/negative parts
is the solution to \cref{eq:lp},
observe that if 
$z_{\tilde{j}}:=\min\{x^{+}_{\tilde{j}},x^{-}_{\tilde{j}}\}>0$, 
we have that
\crcom{
$$x_{\tilde{j}}
=(x^{+}_{\tilde{j}}-z_{\tilde{j}})-(x^{-}_{\tilde{j}}-z_{\tilde{j}})
=
x^{+}_{\tilde{j}}-x^{-}_{\tilde{j}},$$ with $(x^{+}_{\tilde{j}}-z_{\tilde{j}}),(x^{-}_{\tilde{j}}-z_{\tilde{j}})\ge0 $ is an admissible decomposition }
(i.e., $\vec{A} (\vec{x}^{+}-\vec{x}^{-})=\vec{b}$) as well, where one term vanishes.
\end{proof}

We note that the direct application of \cref{eq:lp} to \cref{eq:shape_function} would
use, for fixed $z$, $\vec{A} :=V_z$, $\vec{b}=p_z$ and $\vec{x}= a^*(\cdot,z)$.
Thus $m= \dim(\mathcal{P}_k(\M))$ and $n= \# (X\cap B(z, \varTheta_k h_X /{\epsilon} ))$.

\crcom{
As specific methods to solve the linear program we point out the iterative algorithm presented in \cite{Wright1996}. }
\thcom{With a user-determined} \crcom{ value for the duality measure}
 \thcom{$\mu>0$,}
\crcom{ this  method produces after $K=\mathcal{O}(n^2 \log(\mu^{-1}))$ iterations
an $a^{(k)}(\cdot,z)$ for $k\ge K$  which is primal admissible. i.e., satisfies the
constraint from (\ref{eq:shape_function}), namely $V_z a^{(k)}(\cdot,z)=p_z$, and which satisfies an error bound
\begin{align*}
    \|a^{(k)}(\cdot,z)\|_{\ell_1{(X_z  )}} \le \|a^{\ast}(\cdot,z)\|_{\ell_1{(X_z)}}+\sqrt{n}\mu.
\end{align*}
This estimate follows by combining \cite[Corollary 3.7, Theorem 5.5]{Wright1996} with an application of  H\"older's inequality.}
\thcom{This is an efficiently implementable alternative to the
$\ell_1$-minimizer in that 
the choice $\mu=n^{-1/2}\epsilon$
permits it to satisfy  \cref{ass:A_approximator}. Indeed, by following the
proof of \cref{{lem:lemma_assumption2}},
polynomial reproduction guarantees that  the perturbed operator
$$\widetilde{\MLAD}:\vec{c}\mapsto \sum_{j=1}^N c_{j} a^{(k)}(x_j,\cdot)$$ 
enjoys
 approximation rate
 $\|f-\widetilde{\MLAD}(f|_X)\|_{\infty}
 \le C (h_X/\epsilon)^{k+1} \|f\|_{C^{k+1}(\M)}$,
 while the above error bound
 assures $  \|a^{(k)}(\cdot,z)\|_{\ell_1{(X_z  )}} \le 1+2\epsilon$.
Note that for quasi-uniform point sets,
we have $n\sim \epsilon^{-d_{\M}}$,
which yields an iteration count of
$K=\mathcal{O}(\epsilon^{-2d_{\M}}|\log \epsilon|)$.
There are also interior point methods with different cost complexities, namely $K=
\mathcal{O}(n^c \log(\mu^{-1}))
$ for $c\le 2$ (and thus
$K=
\mathcal{O}(\epsilon^{-c d_{\M}} \log(\epsilon))$ in case
of quasi-uniformity)
but the error analysis gets more involved in those cases, see \cite{Wright1996}. 
}
%
%
%
\section{Numerical methods for ODEs on manifolds}
\label{sec:ODE}
In this section, we discuss a
strategy for satisfying \Cref{ass:A_integrator}: obtaining a high-order 
ODE solver for embedded manifolds. Our approach is a variant of the projection method, as 
considered in \cite[Algorithm 4.2]{HairerODE},
and used in this context in \cite{PLR}.
We use a one-step Runge-Kutta (RK) solver in the ambient space,
which produces a point in \gwcom{$\R^d$} near to the manifold,
 and then maps the computed solution onto the manifold via a 
suitable {\em retraction}
$r$ (described below). In short, we consider a method which looks like
$$
y_n 
\qquad \rightsquigarrow  
 \qquad 
\tilde{y}_{n+1} = \Phi_h (y_n)
\qquad
\rightsquigarrow  
\qquad
 y_{n+1}=r( \tilde{y}_{n+1} ).$$ 
A fundamental requirement of this method is that the retraction produces an error on par with the one-step solver.

A basic version of this method
 selects
 $\mathrm{proj}_{\M}(x) =\mathrm{argmin}_{z\in\M} |x-z|$,
 the closest  point from the manifold.
As shown in \cite{PLR}, the closest point retraction  yields a one step method which performs as well, i.e., with the same approximation rates,
as the Euclidean version 
from which it derives. 
However,
determining the closest point  can pose a challenging task.
In this section, we make a modest relaxation to this approach which  avoids this substantial problem by allowing more general
retractions while  delivering similar theoretical results.
%
%
%
\subsection{Tubular neighborhood}
Let $\mathbb M \subset \R^{d}$ be a smooth compact embedded $d_{\M}$-dimensional submanifold.
Here, we follow \cite[Chapter 6]{Lee}. 
We consider the normal bundle
$$N\M=\left\{ (\xi, v) \in \R^{d} \times \R^{d} \ : \ \xi \in \M, \ v\in N_{\xi}\M \right\}\subset \R^{d} \times \R^{d},$$
where $N_{\xi}\M $ denotes the normal space which is a $(d-d_{\M})$ dimensional vector space such that 
$\R^{d}\cong \MLAD_{\xi}\R^{d}= \MLAD_{\xi} \M \oplus N_{\xi}\M$, where $\MLAD_{\xi}\M\subset \R^{d}$ denotes the tangent space.
The map
$$ E: N\M \to \R^{d}, (\xi,v)\mapsto E(\xi,v)=\xi+v$$
is smooth, and
by \cite[Theorem 6.24]{Lee},
 for some $\delta>0$,
 $E$ is a diffeomorphism when restricted to the subset 
 $V_{\delta}=\left\{(\xi,v) \in N\M \ : \ |v|  < \delta \right\}$.
This 
defines 
a tubular neighborhood  $U_{\delta}^\M$: 
$$\M \subset U_{\delta}^\M:=E(V_{\delta}) \subset \R^{d}.$$
Clearly if $0<\delta'\le \delta$, then $U_{\delta'}^\M$ is also a tubular neighborhood. 
The maximum value $\delta$ for which  this holds (i.e., for which $E$ restricted to $V_{\delta}$ is a diffeomorphism)
 is called the {\em reach} of $\M$, as initially studied in \cite{Federer}.
Estimating the size of the reach in terms of local and global geometric quantities is 
still actively investigated \cite{Niyogi,Aamari}.
%
%
%
\subsubsection*{Retractions:}
A retraction is a  map $r: U_{\delta}^\M\to \M$, defined for some $\delta>0$, for which $r|_{\M} = \mathrm{id}$.
The  algorithm  below requires a numerical realization of a retraction to $\M$.

\noindent An example of a retraction is the map $\mathrm{proj}_{\M}$, which produces the closest point from the manifold. For 
an embedded manifold with reach $\tau$,
$$
	\mathrm{proj}_{\M}:U_{\tau}^\M \to \M:  \xi \mapsto \arg\min_{\zeta \in \M} |\zeta-\xi|.
$$
In the case that $\M \subset \R^d$ is a  smooth embedded manifold without boundary, $\mathrm{proj}_{\M}$
is smooth, since it can be expressed as the composition of smooth maps $\mathrm{proj}_{\M} = \pi \circ E^{-1}$, where 
$\pi: N\M \to \M$ is the natural projector and $E$ is defined above. 

In some cases, namely spheres, ellipsoids, and tori, $\mathrm{proj}_{\M}(x)$ can be computed by an appeal to elementary geometry.
For certain matrix manifolds (orthogonal groups, Stiefel and Grassmannian manifolds), explicit formulas for $\mathrm{proj}_{\M}$ exist
via matrix factorizations and by employing the SVD -- see \cite{Absil}.
In the case that $\M$ is an algebraic variety, the set of critical points of $\mathrm{proj}_{\M}$ is a finite set  (of constant cardinality!).
This finite set can be computed as variety of the so-called critical ideal.
This is described in \cite[Lemma 2.1]{Draisma:etal:2016}; see also \cite{Julia} for some worked out examples which 
calculate both the reach of an algebraic variety as well as the closest point.

For a general manifold given as the zero set of smooth functions, 
the approach of \cite{Macdonald} provides an algorithm, but requires access to the smooth function
$x\mapsto \mathrm{min}_{z\in \M} |z- x|^2$, which is, in general, not easily computable (as mentioned there, it can be obtained by solving a
nonlinear eikonal equation, see  also \cite[Theorem 3.1]{Ambrosio:Soner:1996}).
This is the motivation for considering more general retractions in \Cref{sec:MbRK} -- we note that the method presented in \cite[Section 5]{Macdonald} produces these under fairly general conditions.
The result given in \Cref{sec:ncp} shows that such retractions produce near optimal points, which is a requirement of \Cref{lem:L_MbRK} below.
%
%
%
\subsection{Manifold-based Runge-Kutta}
\label{sec:MbRK}
We consider a dynamical system 
\begin{equation*}
	\begin{aligned}
		{\dot{\zeta}(s)}&{= g(s,\zeta(s)) },&
		\zeta(0)&=\zeta_{0} \in \M,
	\end{aligned}
\end{equation*}
were 
${g: (-\epsilon,\epsilon)\times \M \to \R^d}$ 
with 
${g(s,\zeta)\in \MLAD_{\zeta} \M \subset \R^{d}}$
i.e., a vector field. 
Moreover, we assume the vector field to be smooth in the sense of 
$g(s,\cdot) \in C^{\ell}(\M, \R^{d})$.
Then there is a maximal interval $I_{\zeta_{0}} \subset \R$ such that 
$\zeta: I_{\zeta_{0}} \to \M$ satisfies the differential equation. 
Finally, $\zeta$ is unique in the sense that every other solution $\tilde{\zeta}:\tilde{I}_{\zeta_{0}}\to \M $ with $\tilde{I}_{\zeta_{0}}\subset  I_{\zeta_{0}} $  satisfies $\zeta_{\tilde{I}_{\zeta_{0}}}=\tilde{\zeta}_{\tilde{I}_{\zeta_{0}}}$.
For a retraction $r:U_{\delta}^{\M}\to \M$, we consider the extension 
\begin{equation*}
	{G:(-\epsilon,\epsilon)\times U_{\delta}^\M \to \R^N:(s,z) \mapsto g(s,r(z))}
\end{equation*}
which leads to the following dynamical system
\begin{equation}
\label{eq:system}
	\begin{aligned}
		{\dot{z}(s)}&{= G(s,z(s)) }&,
		z(0)&=\zeta_{0} \in \M.
	\end{aligned}
\end{equation}
The retraction property $r|_{\M}=\operatorname{id}$ ensures with the uniqueness of solutions that $z:I_{\zeta_{0}}\to \M$ and $z|_{I_{\zeta_{0}}}=\zeta|_{I_{\zeta_{0}}}$. In particular, we have $z(t) \in \M$ for $t\in I_{\zeta_{0}}$.

For the numerical solution, we consider a RK method \crcom{$\tilde{\Phi}_{RK}$ in the ambient space}. 
Let $\alpha_{\ell}, \beta_{j,\ell}, \gamma_{\ell}$, $\ell=1,\dots,m$ denote the weights for an $m$-step RK method, typically given in form of a Butcher's tableau.
We denote by $\eta_i$ the approximation to the true solution $z(t_i)=\zeta(t_i)$.
We set $t_{i+1}-t_{i} = h_{t}$, $\eta_0=z(0)=\zeta_0$ and define, recursively,  for $i\in \N$
\begin{equation}\label{eq:RK}
	\begin{aligned}
		k_{j}(t_i, \eta_{i}, h_{t}) &= {G} \left( t_i+\alpha_j h_t, \eta_{i} + h_{t} \sum_{b=1}^{m} \beta_{j,b} k_{b}(t_i,\eta_i,h_{t})  \right) \quad j=1,\dots,m\\
		\hat{\eta}_{i+1}&=\eta_{i} + h_{t} \sum_{b=1}^{m} \gamma_{b} k_{b}(t_{i},\eta_{i},h_{t}) = \eta_{i} +h_{t} \crcom{\tilde{\Phi}_{RK}}(t_i,\eta_i,h_{t})\\
		\eta_{i+1}&=r\left(\hat{\eta}_{i+1} \right),	
		\end{aligned}
\end{equation}
where we employ
$\eta_{i} \in \M $ and hence $r(\eta_{i})=\eta_{i}$. \crcom{This yields in particular 
$$\tilde{\Phi}_{RK}(t,\eta,h)=\sum_{b=1}^{m} \gamma_{b} k_{b}(t,\eta,h)$$.}
If $\beta_{j,b}=0$ for $b \ge j$, the RK method is called explicit, otherwise it is  implicit.
To ensure that the retraction is well-defined
(i.e., $\hat{\eta}_{i+1} \in U_{\delta}^\M$ for each $i$)
we require
that the time step $h_t$ satisfies
\begin{equation}\label{eq:welldefRK}
	 h_{t} \left| \sum_{b=1}^{m} \beta_{j,b} k_{b}(t_i,\eta_i,h_{t})\right| < \delta,
\end{equation}
where the coefficients are from \cref{eq:RK}.
In case of implicit RK methods, we \crcom{had additionally to assume 
	 $ h_{t} \le \left(L^{(2)}_{\bar{G}}\max_{i=1}^{m}\sum_{j=1}^{m} |\beta_{i,j}| \right)^{-1}$, where $L^{(2)}_{\bar{G}}>0$ denotes the Lipschitz constant of $\bar{G}$ with respect to the second argument, see e.g. \cite[Theorem 14]{Butcher}. But we do not consider those schemes here.}
Implicit RK methods allow for simple consistency criteria based on polynomial exactness for quadrature.
For explicit RK methods, the choices for weights are typically more involved, see e.g., \cite[Sec 7.3]{HermannSaravi}.
The next lemma shows that a consistent global RK method (given in \cref{eq:consistencyunconstrianed}) leads to consistency of the method \cref{eq:RK}, see also \cite{PLR} for such a statement.
\begin{lemma}
\label{lem:L_MbRK}
Consider 
\cref{eq:system} with $G(t,\cdot) \in C^{\ell}(U_{\delta}^{\M})$ for each $t\in(-\epsilon,\epsilon)$.
Suppose that timestep $h_t$ is small enough to satisfy \cref{eq:welldefRK} and that the RK method \cref{eq:RK} provides consistency order $\ell$:
\begin{equation}
    \label{eq:consistencyunconstrianed}
    \left|z(t+h_t) - \left(z(t)+h_t \crcom{\tilde{\Phi}_{RK}}\left(t,z(t),h_t) \right) \right) \right|\le \CRK h^{\ell+1}_{t},
\end{equation}
\crcom{where $\tilde{\Phi}_{RK}$ encodes the ambient space RK method introduced in \cref{eq:RK}.}
If there is a constant $\Cret$ so that $r:U_{\delta}^{\M}\to \M $ is a retraction satisfying the {\em near closest point} property $|x-r(x)|\le \Cret \min_{z\in\M} |x-z|$ for all
$x\in U_{\delta}^{\M}$, then
\begin{equation}
    \label{eq:consistency}
    \left|z(t+h_t) - r \left(z(t)+h_t \crcom{\tilde{\Phi}_{RK}}\left(t,z(t),h_t) \right) \right) \right|\le \CRK(\Cret+1)  h^{\ell+1}_{t}.
\end{equation}
\end{lemma}

\begin{proof}
    We point out that in general $z(t+h_t)\approx  z(t) +h_{t} \crcom{\tilde{\Phi}_{RK}}(t,z(t),h_{t})  \not \in \M$, though $z(t)=\zeta(t) \in \M$ for all $t\in I_{\zeta_{0}}$. 
    The condition \eqref{eq:welldefRK}, however, makes sure that $z(t) +h_{t} \crcom{\tilde{\Phi}_{RK}}(t,z(t),h_{t}) \in U_{\delta}^\M$. Let $  \varepsilon(t):=
	\left|z(t+h_t)- r\left( z(t) +h_{t} \crcom{\tilde{\Phi}_{RK}}(t,z(t),h_{t})  \right) \right|$
    denote the one-step error.
By applying the triangle inequality, we observe
that
\begin{align*}
	\varepsilon(t)
    &\le \left|z(t+h_t)-(z(t) +h_{t} \crcom{\tilde{\Phi}_{RK}}(t,z(t),h_{t}) ))\right|\\
	&+\left| (z(t) +h_{t} \crcom{\tilde{\Phi}_{RK}}(t,z(t),h_{t}) ))- r\left( z(t) +h_{t} \crcom{\tilde{\Phi}_{RK}}(t,z(t),h_{t})  \right) \right|\\
	&\le \left|z(t+h_t)-(z(t) +h_{t} \crcom{\tilde{\Phi}_{RK}}(t,z(t),h_{t}) )\right|\\&+
	\Cret\left| (z(t) +h_{t} \crcom{\tilde{\Phi}_{RK}}(t,z(t),h_{t}) ))-z(t+h)\right|\\
	&= (\Cret+1) \left|z(t+h_t)-(z(t) +h_{t} \crcom{\tilde{\Phi}_{RK}}(t,z(t),h_{t}) ))\right| \le \CRK(\Cret+1)  h^{\ell+1}_{t},
\end{align*}
where we used the closest
point property of $r$, i.e., $\left| x -r(x)\right| \le  \Cret\left| x -y \right|$ for all $y \in \M$ and the fact that $z(t+h) \in 
\M$ is admissible. 
\end{proof}
This shows that the conditions in \cref{eq:One_step_error} can be met by conventional RK methods. 
\crcom{
This means for \Cref{ass:A_integrator}, that we define}
$$
 \thcom{  \Phi_{(x,t)}(h):= r(x+h\tilde{\Phi}_{RK}(t,x,h)).}
$$
If we consider a bounded subset $\Omega \subset \M$, we need to make sure that the solutions stay in $\Omega$, i.e., $\Omega$ is flow-invariant. Sufficient conditions for flow-invariance are known in the Euclidean setting, see for example \cite{Redheffer,Ladde:Lakshmikantham:1974}.

As in \cite[Eq. 4.1]{FalconeFerretti}, we consider a Bony-type condition.
This requires in particular \crcom{$x \mapsto G(x)$} to be Lipschitz continuous.  
\crcom{Moreover, we require $\M$ to be a $C^2$ manifold and $\Omega \subset \M$ to be a closed set, which allows to define an exterior normal set $\thcom{\mathfrak{N}_{e}}(\Omega) \subset T^{\star}\M \setminus \{0\}$, see \cite[Definition 8.5.7]{Hormander1}.
Then, flow invariance is ensured by
\begin{equation}
\label{eq:flowinvariance}
    (G(x),\nu) <0 \quad \text{for all } (x,\nu)\in \mathfrak{N}_{e}(\Omega),
\end{equation}
see \cite[Theorem 8.5.11]{Hormander1}. 
An inspection of the proof reveals that this holds also for non-autonomous systems if $x \mapsto G(s,x)$ is now Lipschitz independently of $s$ and $(G(s,x),\nu) <0$ for all $(x,\nu)\in \thcom{\mathfrak{N}_{e}(\Omega)}$ and all $s\in [0,t_m]$.
}
%
%
%
\subsection{Near closest point property of smooth retractions}
\label{sec:ncp}
We now show that the retraction property is sufficient to guarantee 
that it produces {\em near-closest points}.
\begin{lemma}
If $r:U_{\delta}^\M\to \M$  is a $C^1$ retraction then there exists $\Cret$ 
 and $\delta'>0$  so that if $z\in U_{\delta'}^\M$, then 
 $$|z-r(z)| \le \Cret |z-\mathrm{proj}_{\M}(z)|.$$
\end{lemma}
\begin{proof}
For $z\in \M$, let $O_z\subset \M$  be a neighborhood so that  $\psi_z:O_z\to \R^d$ is a chart.
Find a  neighborhood ${W_z} \subset r^{-1}(O_z) \subset U_{\delta}^\M$ of $z$
where $\M\cap W_z$ is a level set of a defining function
$\Lambda_z: W_z\to \R^{d-d_{\M}}$,
 guaranteed to exist by \cite[Proposition 5.16]{Lee}.
Thus  $\M\cap W_z =\Lambda_z^{-1}(0)$
and    at every point $x$ of  
$\M\cap W_z$,
$d \Lambda_z(x)$ is surjective.
Define  $\Psi_z: {W_z} \to \R^d$ by
 $$\Psi_z(y) =  \bigl(\psi_z(r(y)), \Lambda_z(y)\bigr).$$
 We now show that $d\Psi_z(z):\MLAD_z\R^d\to \MLAD_z\R^d$ is nonsingular.
 
 For $v\in  \MLAD_z\R^N= \MLAD_z\M \oplus N_z\M$, write  $v=v^{\|} +v^{\perp}$, with $v^{\|}\in \MLAD_z\M$
 and $v^{\perp}\in N_z\M$.
 Note that $d\Psi_z(z) v=0 $ ensures that 
 $d\Lambda_z(z) v = d\Lambda_z(z) v^{\perp}=0$ since $d\Lambda_z(z) v^{\|} =0$ by the level set property.
But  the restriction of $d\Lambda_z(z)$ to
the normal space $N_z\M$ is an isomorphism: i.e.,  $d\Lambda_z(z)|_{N_z\M}: 
N_z\M\to \R^{d-d_{\M}}$  is injective.
 Thus  $v^{\perp} =0$.
 
 The differential $dr(z):\MLAD_zU_{\delta}^\M \to \MLAD_z\M$ is surjective, since each $v\in \MLAD_z\M$ is
 a tangent vector in $\MLAD_zU_{\delta}^\M$ and $dr(z) v = v$, since $r|_{\M} = \mathrm{id}$.
 By the chain rule, the differential of $\psi_z\circ r$ at $z$ is $d\psi_z(r(z))dr(z)$.
 Thus  $d\Psi_z(z) v^{\|}=0$, which implies that
 $d\psi_z(r(z))dr(z) v^{\|}=0$, which implies $dr(z)v^{\|}=0$, and so $v^{\|}=0$.

By the inverse function theorem, there is a neighborhood $W_z^{\flat} \subset U_{\delta}^\M$  of $z$ 
on which $\Psi_z:W_z^\flat\to \Psi(W_z^\flat)$ is a $C^1$ diffeomorphism. Consequently, there are
constants $0<C_1(z)\le C_2(z)$ so that  for any $x,y\in W_z^\flat$, 
$$C_1(z) |x-y| \le |\Psi_z(x)-\Psi_z(y)|\le C_2(z)|x-y|.$$ 
Define $W_z^{\sharp}:=\{x\in W_z^\flat\mid r(x)\in W_z^\flat\}$. 
Since this is a neighborhood of $z$,
there is  $R_z>0$  so that $B(z,2R_z)\subset  W_z^\sharp$.

By compactness, there exist $z_1\dots z_M\in\M$ so that  $\M\subset  \bigcup_{j=1}^M B(z_j,R_{z_j})\subset U_{\delta}\M$.
Since this is a neighborhood of $\M$, there is
a tubular neighborhood $U_{\delta'}^\M \subset \bigcup_{j=1}^M B(z_j,R_{z_j}).$
Let $C_1:=\min(C_1(z_j))$ and let $C_2:=\max(C_2(z_j))$.

For any $x\in U_{\delta'}^\M$, $x$ is in some $B(z,R_z)\subset W_z^\sharp$,
and thus $r(x)\in W_z^\flat$.
We have also that 
$\mathrm{proj}_{\M}(x)\in W_z^\flat$, since   $|\mathrm{proj}_{\M}(x)-x|\le |x-z| <R$,
which implies that
$\mathrm{proj}_{\M}(x) \in B(z,2R)$.
Finally, because $x,r(x)$ and $\mathrm{proj}_{\M}(x)$ are in $W_z^\flat$, we have that
$$|x-r(x)| \le \frac{1}{C_1}  |\Psi_z(x) -\Psi_z(r(x))| = \frac{1}{C_1} |\Lambda_z(x) -\Lambda_z (r(x))| =  \frac{1}{C_1}|\Lambda_z(x)|,$$
while $\Lambda_z(\mathrm{proj}_{\M}(x)) =0$ implies
 $|x-r(x)|  \le  \frac{1}{C_1} |\Lambda_z(x) -\Lambda_z (\mathrm{proj}_{\M}(x) )|$.
Thus
$$|x-r(x)| 
\le
\frac{1}{C_1} |\Psi_z(x) -\Psi_z (\mathrm{proj}_{\M}(x) )|
\le
  \frac{C_2}{C_1}|x-\mathrm{proj}_{\M}(x)|.$$
Thus the lemma holds with $\Cret = \frac{C_2}{C_1}$.
\end{proof}
We note that \cite[Section 5]{Macdonald} presents a constructive method for producing smooth retractions
(called {\em closest point functions} by the authors) when $\M$ is given as a level set. \crcom{It is possible} to construct $\mathrm{proj}_{\M}$ (called the {\em Euclidean closest point function} by the authors) 
when $\M$ is a hypersurface, if one has access to the signed distance to $\M$.
%
%
%
\section{Numerical Results}
\label{sec:NumericalResults}
We consider several numerical tests to illustrate the theory developed in the preceding sections. The first set of tests focus on the Lebesgue constants of the new remapping operator $\MLAD_{X}$ discussed in \Cref{sec:ell1} and compares these to the more standard MLS spatial approximator $\MLS_{X}$ in \cref{eq:Pol_MLS}. The next experiments focus on SL advection in an autonomous velocity field on the torus.  The final test explores the SL scheme when the remapping operators do not satisfy \Cref{ass:A_approximator}, to demonstrate that there may be more to the stability of these schemes than the theory developed here.  

For the SL advection tests we use explicit RK methods of various orders $\ell$ for the time integrator. The polynomial precision of the operators $\MLAD_{X}$ and $\MLS_X$ used in the experiments are given by the parameter $k$.

For point sets $X$ on the sphere, we use quasiuniform maximal determinant points of varying size $N$ available from~\cite{Wright_SpherePts_2018}, while for the torus we use point sets generated from DistMesh~\cite{DistMesh}. Both of these points have fill-distances that satisfy $h_{X}\sim 1/\sqrt{N}$.  Examples from each set are shown in \Cref{fig:nodesets}.  For the sphere, the cardinalities of the point sets for the experiments are 5184, 7396, 11449, 21025, and 36864, while for the torus they are 5312, 7520, 11600, 21056, 47320.

\begin{figure}[htb]
\centering
\begin{tabular}{cc}
\includegraphics[width=0.35\textwidth]{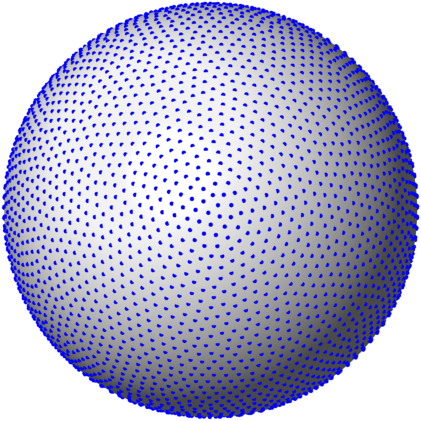} & \includegraphics[width=0.45\textwidth]{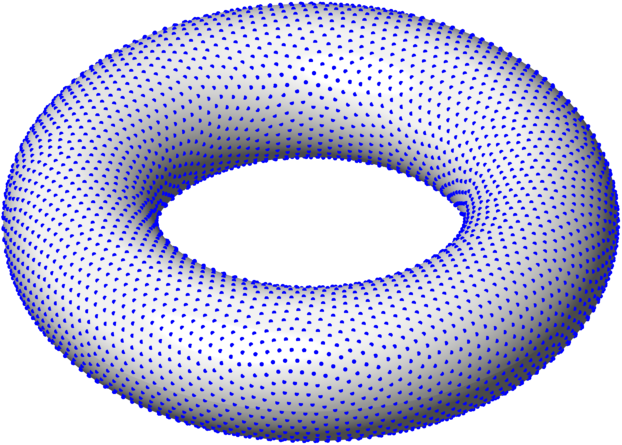}  \\
(a) & (b)
\end{tabular}
\caption{Example of the node sets $X$ used in the numerical experiments: (a) $N=4096$ nodes on the sphere and (b) $5312$ nodes on the torus.  The inner radius of the torus is 1/3, while the outer radius is 1.\label{fig:nodesets}}
\end{figure}
%
%
%
\subsection{Lebesgue Constants\label{sec:lebesgue}}

\begin{figure}[htb]
\centering
\begin{tabular}{cc}
\includegraphics[width=0.46\textwidth]{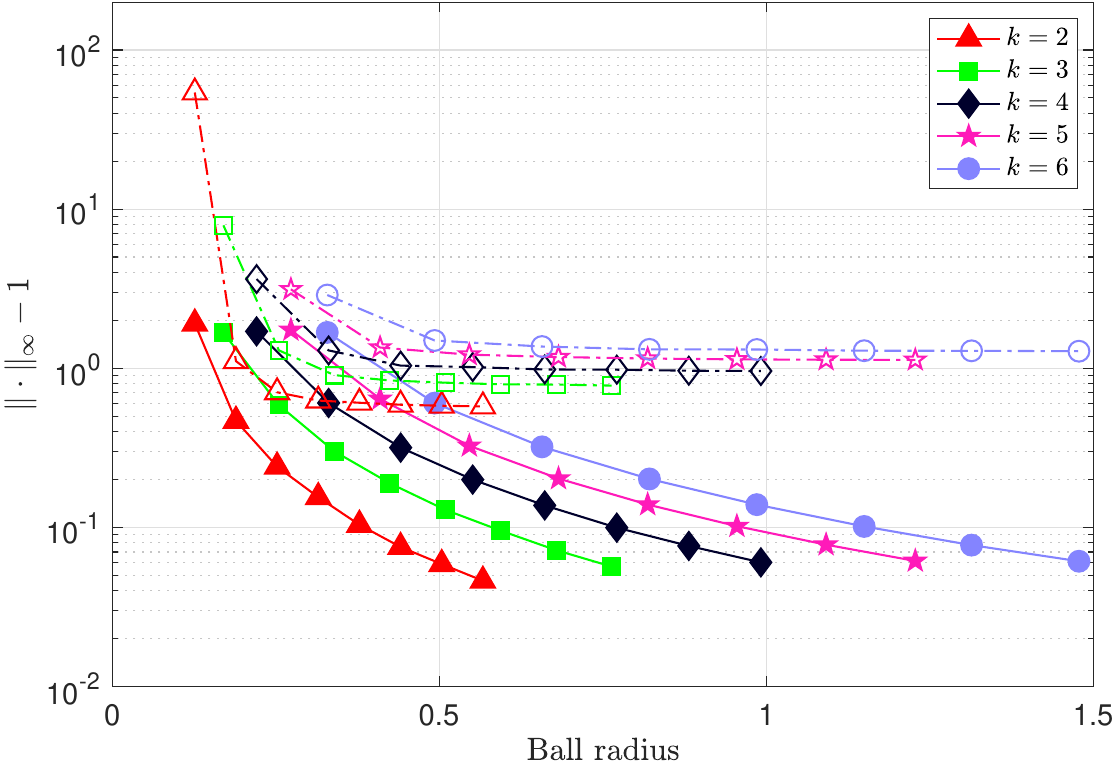} & \includegraphics[width=0.46\textwidth]{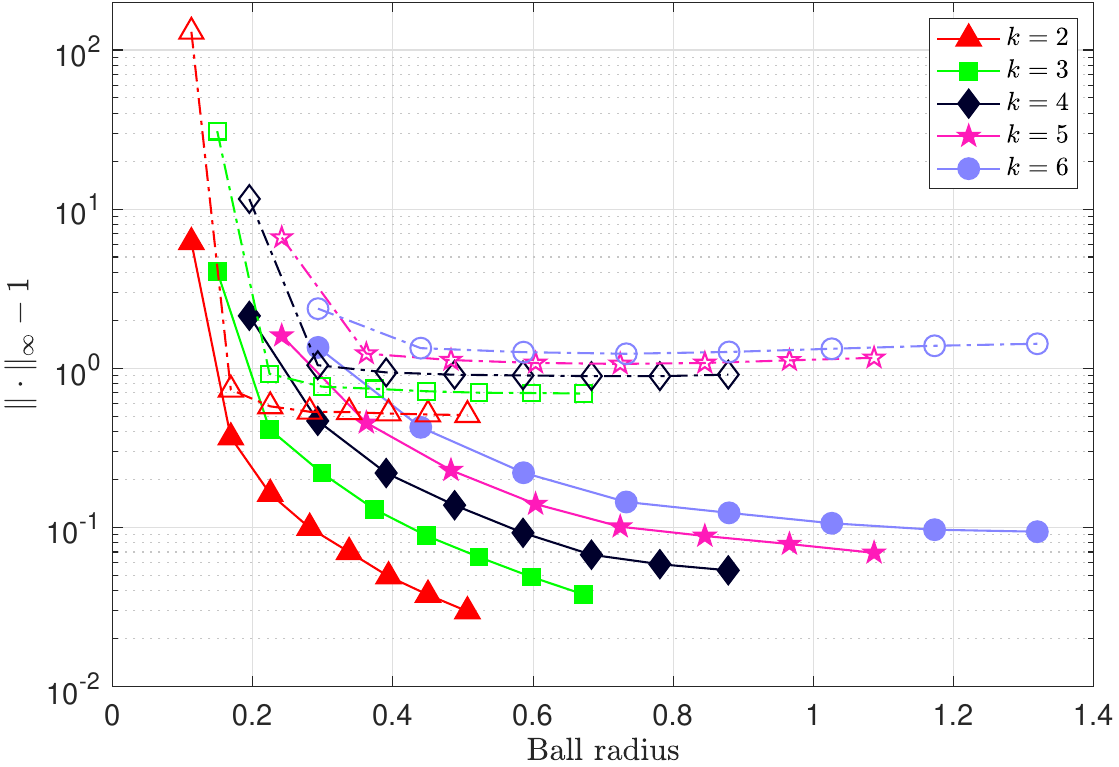} \\
(a) sphere & (b) torus \end{tabular}
\caption{Stability of the remapping  operators $\MLAD_{X}$ (filled markers) and $\MLS_X$ (open markers) for the (a) sphere and (b) torus nodes shown in \Cref{fig:nodesets}.  Result show estimates of the Lebesgue constants of the operators minus 1 (the optimal value) as the radii of support balls increases for different polynomial degrees of precision $k$ and a fixed $h_X$.  \label{fig:solid_body_lebesgue}}
\end{figure}
\begin{figure}[htb]
\centering
\begin{tabular}{cc}
\includegraphics[width=0.46\textwidth]{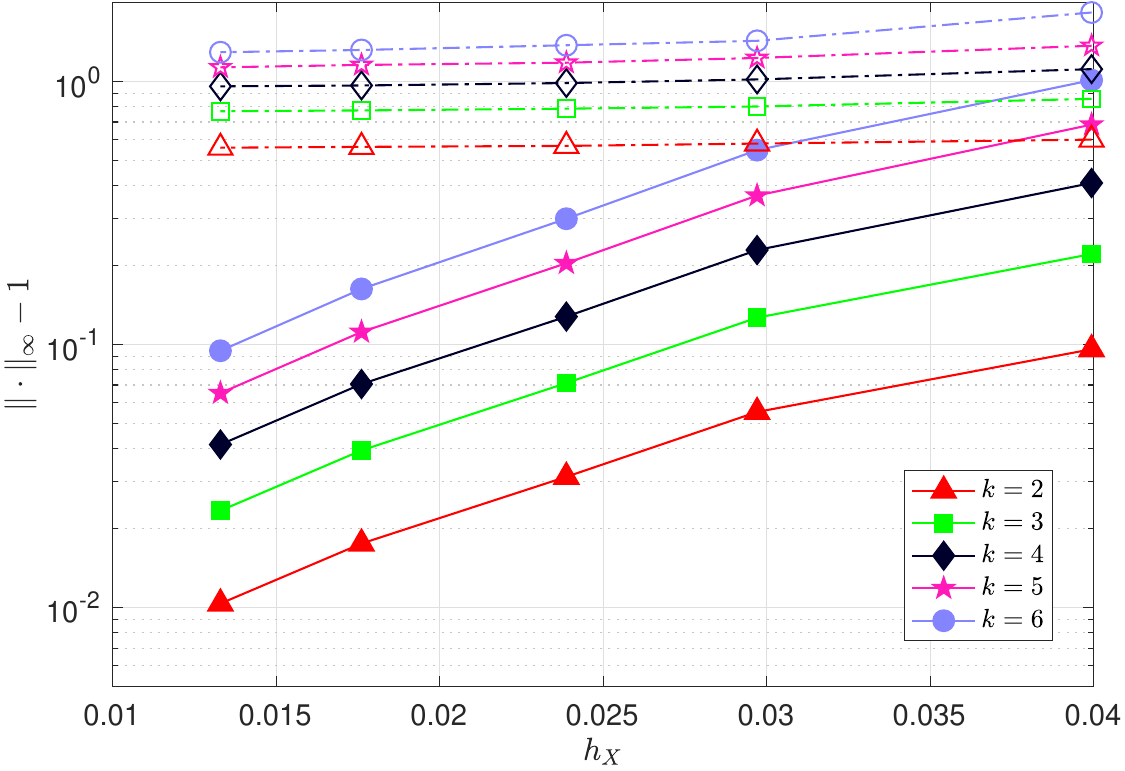} &
\includegraphics[width=0.46\textwidth]{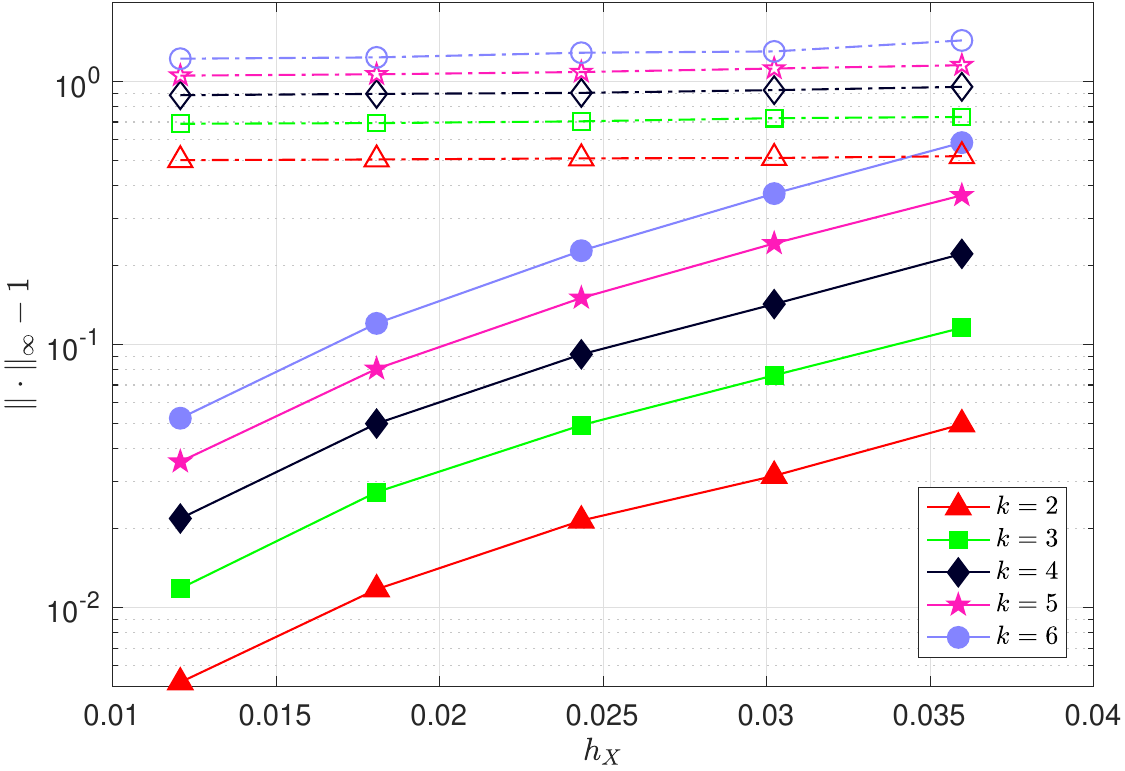} \\
(a) sphere & (b) torus
\end{tabular}
\caption{Stability of $\MLAD_{X}$ (filled markers) and $\MLS_X$ (open markers) for the (a) sphere and (b) torus with radii of the support balls fixed at 0.4.  Result show estimates of the Lebesgue constants of the operator minus 1 (the optimal value) for different polynomial degrees and point sets with decreasing mesh-norm. \label{fig:lebesgue_meshnorm}}
\end{figure}

To numerically study the stability of the approximation operators $\MLAD_{X}$ and $\MLS_X$, we estimate their Lebesgue constants for the two point sets shown in \Cref{fig:nodesets} using samples of the operators on
much denser sets of points on the two surfaces. \Cref{fig:solid_body_lebesgue} shows the results for the cases where the radii of the support balls, $B(\cdot,\Theta_k \frac{h_X}{\epsilon})$, for the operators increase. With $h_X$ and $\Theta_k$ fixed in these tests, increasing the radii is equivalent to decreasing $\epsilon$.  We see in \Cref{fig:nodesets}, that, as expected from \Cref{lem:lemma_assumption2}, $\|\MLAD_{X}\|_{\infty}$ converges towards the value of 1 for all the polynomial degrees tested as the radii increase.  However, for $\MLS_X$ the stability initially improves but then stagnates, indicating that it does not satisfy the first condition of \Cref{ass:A_approximator}.

\Cref{fig:lebesgue_meshnorm} shows a similar result, but where the radii of the support balls are held fixed and the mesh-norm $h_X$ is decreased by increasing the the density of the point sets used to construct $\MLAD_{X}$ and $\MLS_X$. This also produces a decrease in $\epsilon$, but through a different strategy.  We see from the results that $\|\MLAD_{X}\|_{\infty}$ again decreases as expected from \Cref{lem:lemma_assumption2}, but that $\|\MLS_X\|_{\infty}$ stagnates.  
%
%
%
\subsection{Convergence for advection on the torus}

\begin{figure}[htb]
\centering
\includegraphics[width=0.45\textwidth]{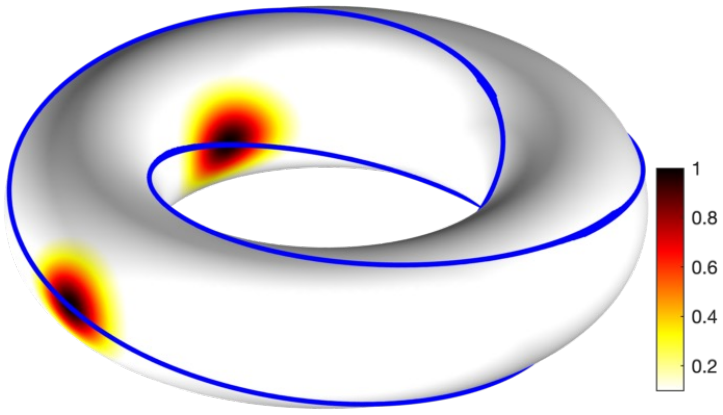}  
\caption{Set-up for the autonomous advection problems on the torus. Initial condition (Gaussian bells) shown as heat map with values corresponding to the displayed color bar. Solid blue lines mark the path of the $(3,2)$ Torus knot that the centers of the bells are transported along during the simulation.  The bells return to their initial positions after integrating to $t=2\pi$. \label{fig:autonomousflow}}
\end{figure}
We consider the advection problem first described in~\cite{shankar2020robust}, which transports the given initial condition along a torus knot~\cite[\S 5.1]{adams2004knot}.  The torus is given by the parameterization
\begin{align*}
    x(\lambda,\varphi) = (1 + R\cos(\varphi))\cos(\lambda),\; y(\lambda,\varphi) = (1 + R\cos(\varphi))\sin(\lambda),\; z(\lambda,\varphi) = R\sin(\varphi),
\end{align*}
where $-\pi \leq \lambda,\varphi \leq \pi$ and $R=1/3$ is the inner radius of the torus.  The $(p,q)$ knot is given by the parameterization
\begin{align*}
x(s) &= (1 + R\cos(q(s-s_0))\cos(ps) \\
y(s) &= (1 + R\cos(q(s-s_0))\sin(ps) \\
z(s) &= -R \sin(q(s-s_0)),
\end{align*}
where $-\pi \leq s \leq \pi$, and $s_0$ determines the initial position of the knot. For a given point in parametric space $(\lambda,\varphi)$ on the tours, the tangential velocity field that follows the $(p,q)$ knot is given with respect to the Cartesian basis as
\begin{align*}
\vec{a}(\lambda,\varphi) = 
\begin{bmatrix}
Rq\sin(\varphi) \cos(\lambda) - p (1 + R\cos(\varphi)) \sin(\lambda) \\
Rq\sin(\varphi)\sin(\lambda) + p(1 + R\cos(\varphi))\cos(\lambda) \\
-Rq\cos(\varphi)
\end{bmatrix},
\end{align*}
which follows from differentiating $x(s),y(s)$, and $z(s)$ and transforming the result to the parameter space of the torus.
As in~\cite{shankar2020robust}, we use $p=3$ and $q=2$.  

The following sum of two Gaussian bells is used as the initial condition:
\begin{align}
v(x,y,z,0) = e^{-20(r_1(x,y,z)))^2} + e^{-20(r_2(x,y,z))^2},
\end{align}
where $r_i(x,y,z)$ is the Euclidean distance from $(x,y,z)$ to $(x_i,y_i,z_i)$, and the centers of the bells are given as $(x_1,y_1,z_1)  = (-1-b,0,0)$ and $(x_2,y_2,z_2)  = (0,1-b,0)$.
See \Cref{fig:autonomousflow} for an illustration of the flow path and the initial condition along the $(3,2)$ knot. For the retraction operator, we use the exact closet point projection.

For the first set of advection experiments, we set the time-step as $h_t = \frac12 \sqrt{h_X}$ and the radii of the balls used to construct $\MLAD_{X}$ proportional to $\sqrt{h_X}$, which means $\epsilon\sim\sqrt{h_X}$.  According to the convergence theory developed in \Cref{sec:main_theorem}  and \Cref{sec:ell1_remap}, we expect the convergence of the SL scheme to be $h_X^{\ell/2} + h_X^{k/2}$, where the first term comes from the time-integration and the second from the remapping operator. \Cref{fig:torus_rk} shows the relative max-norm errors in the computed solution for $\ell,k=2\ldots,6$ as $h_X$ decreases.  The Figure clearly shows which terms in the error are dominant for different combinations of $\ell$ and $k$.  For $\ell=1$ and $2$, we see that the time-integration errors dominate, while for $\ell > 2$, the spatial errors begin to dominate for different $k$.  In all cases, the observed convergence rates are close to or exceed the theoretical estimates.
\begin{figure}[tbh]
\begin{tabular}{cc}
{\footnotesize$\ell=1$} & {\footnotesize$\ell = 2$}\\
\includegraphics[width=0.4\textwidth]{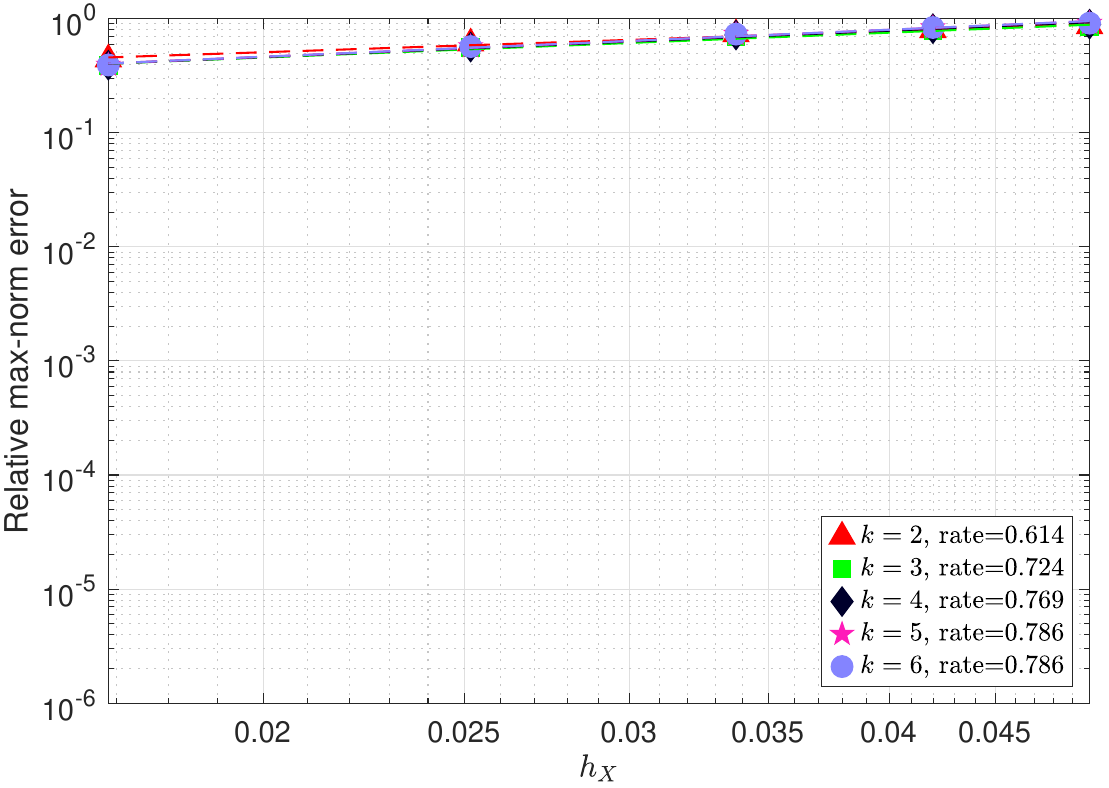} & \includegraphics[width=0.4\textwidth]{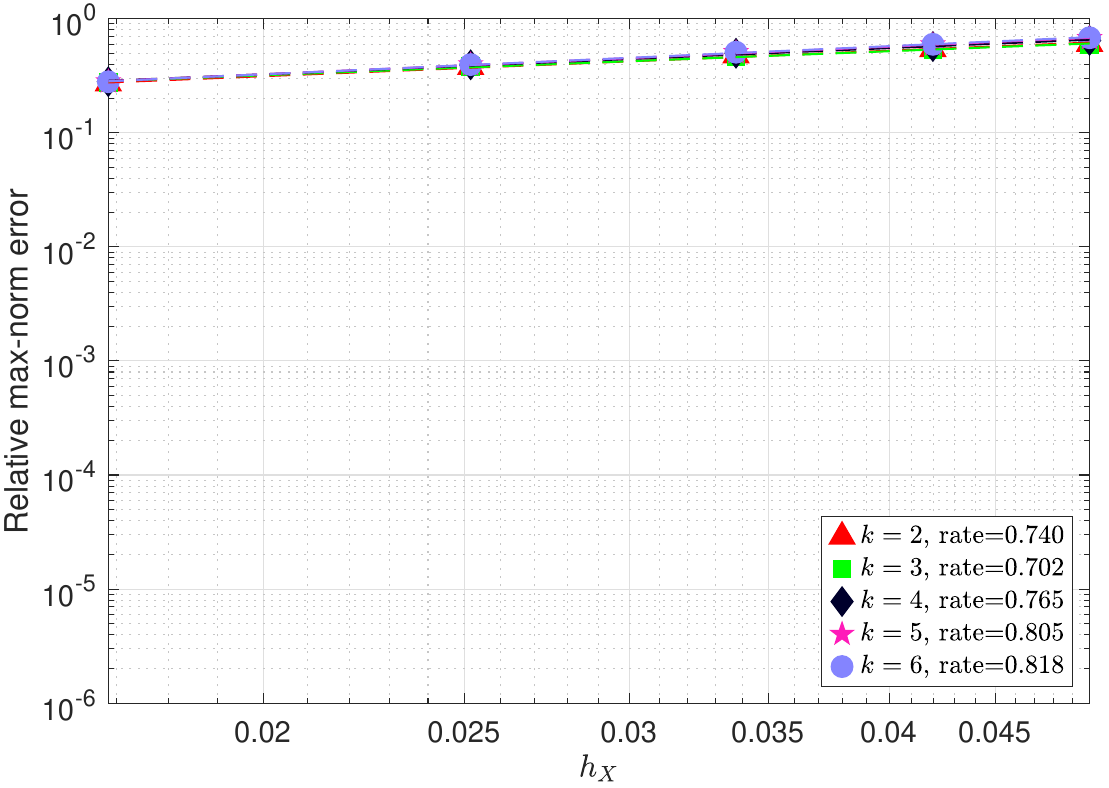} \\
{\footnotesize$\ell=3$} & {\footnotesize$\ell = 4$}\\
\includegraphics[width=0.4\textwidth]{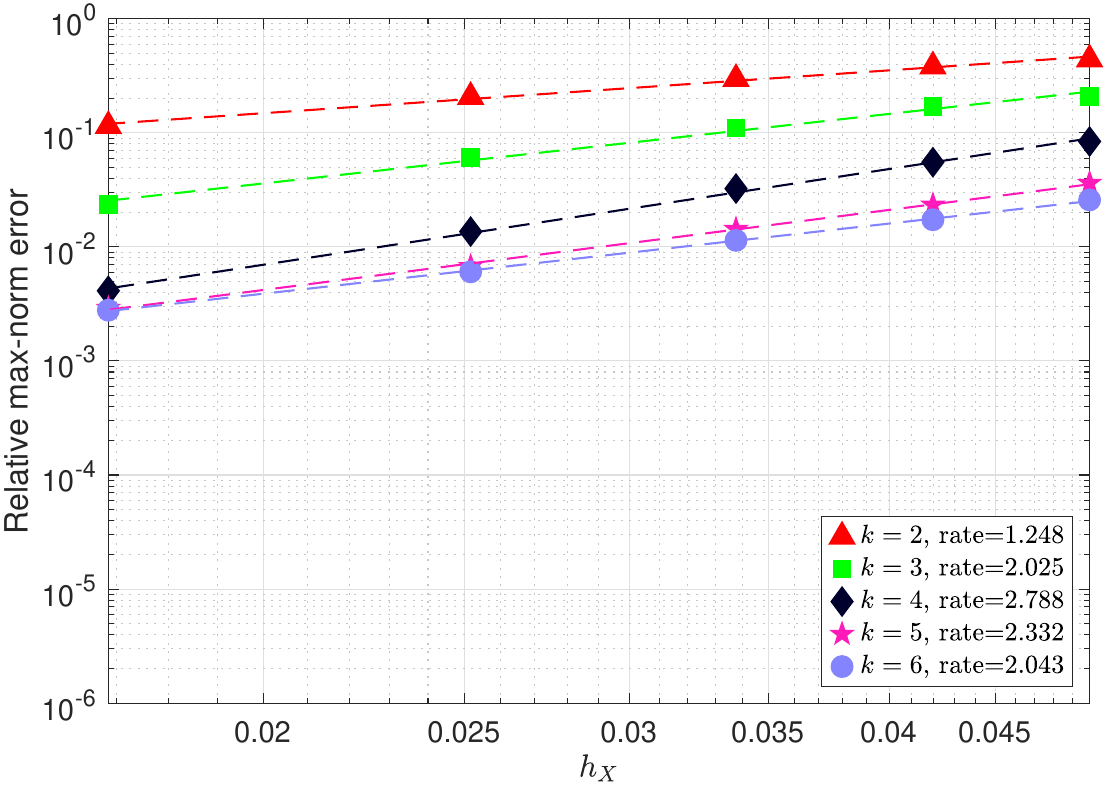} & \includegraphics[width=0.4\textwidth]{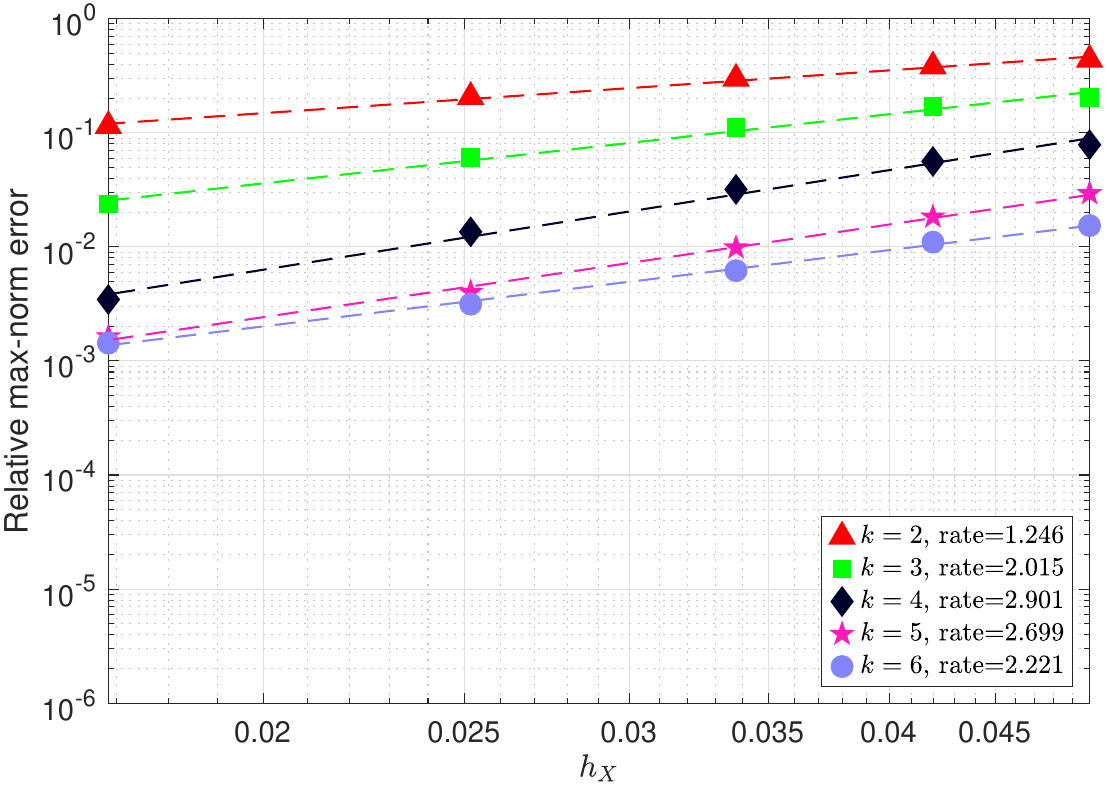} \\
{\footnotesize$\ell=5$} & {\footnotesize$\ell=6$}\\
\includegraphics[width=0.4\textwidth]{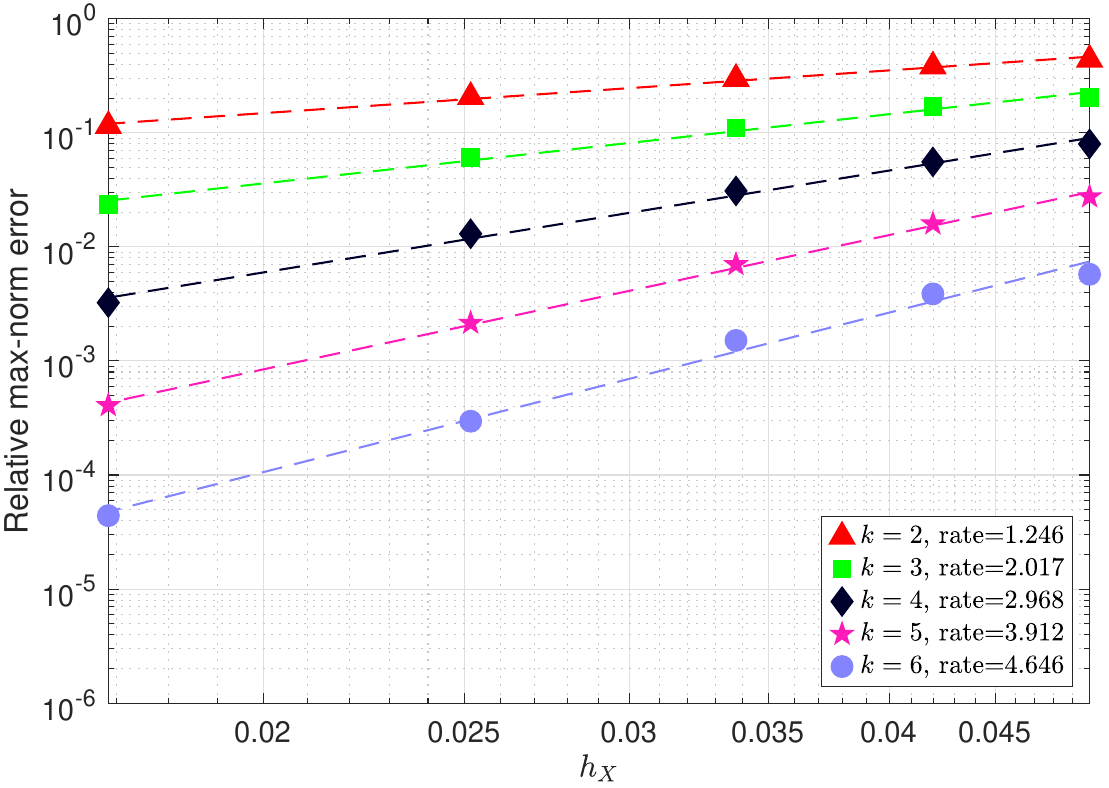} & \includegraphics[width=0.4\textwidth]{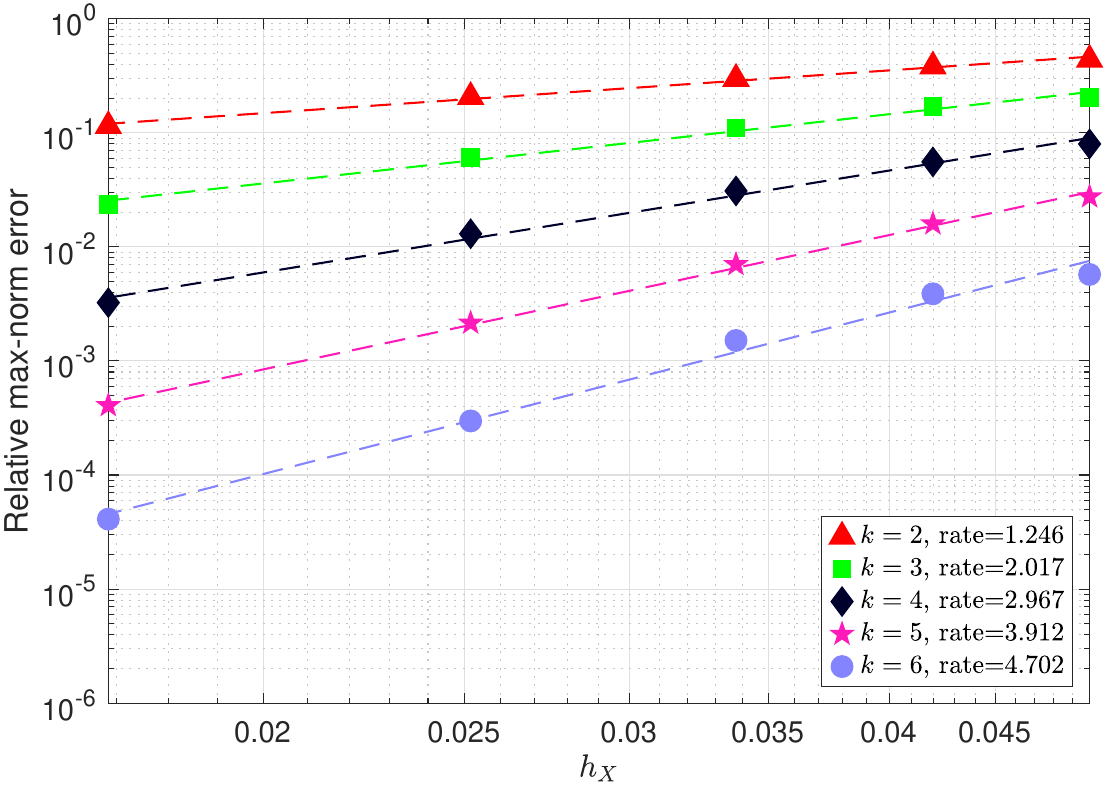}
\end{tabular}
\caption{Convergence results for the SL scheme applied to the torus advection problem for different choices of parameters: $\ell$=order of accuracy of the RK method, and $k$=the polynomial precision of $\MLAD_{X}$.  Dashed lines show lines of best fit according to the estimated convergence rates given in the legend.\label{fig:torus_rk}}
\end{figure}
In the next set of advection experiments, we select the parameters to ensure an overall theoretical convergence rate of $h_X^s$, as in \cref{eq:motivation}: $h_t \sim h_X^{1/\alpha}$, $\ell=\alpha s$, the radii of the balls for $\MLAD_{X}$ proportional to $h_t$, and $k=\lceil(\ell+2-\alpha)/(\alpha-1)\rceil$. \Cref{fig:torus_overall_convg} shows the results of the experiments for different $s$.  We see that the observed convergence rates are all higher than the expected rates.  The two results for $s=2$, also seem to indicate that it can be beneficial to use larger $k$ than $\ell$ to improve convergence, but at the cost of larger radii for the balls used to construct $\MLAD_{X}$.
\begin{figure}[htb]
\centering
\includegraphics[width=0.5\textwidth]{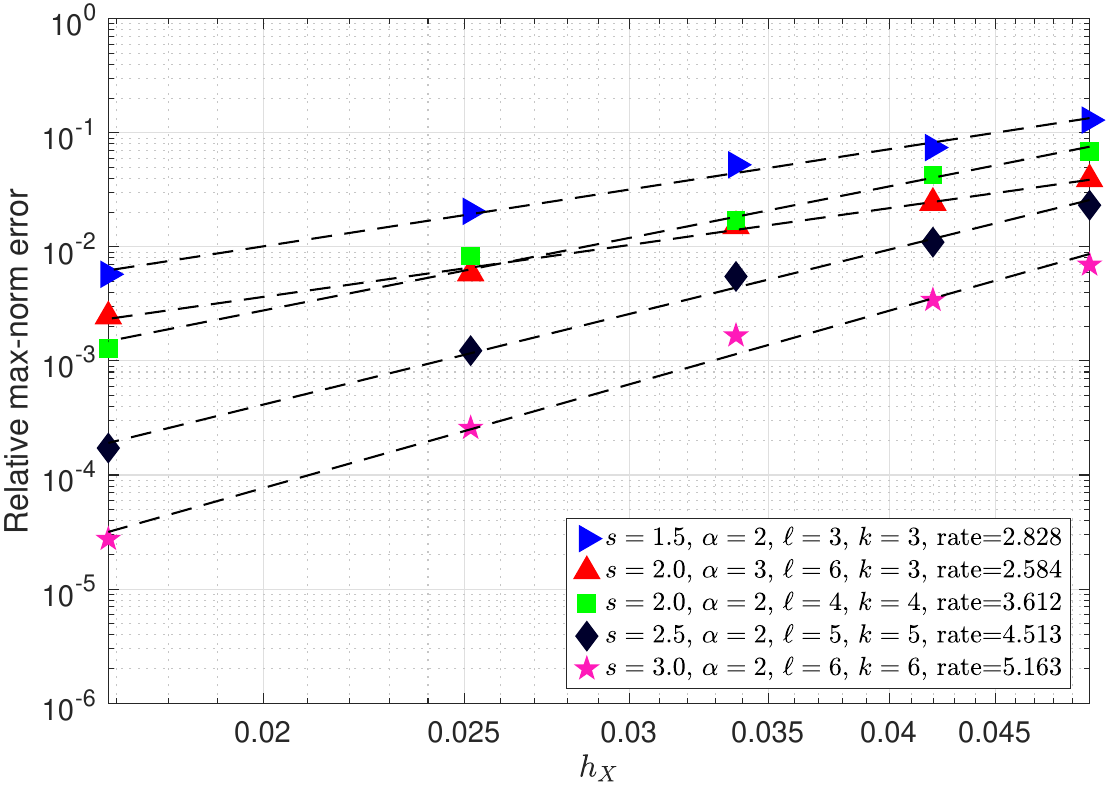}
\caption{Same as \Cref{fig:torus_rk}, but with the parameters selected to give an overall theoretical convergence rate of $h_X^s$.\label{fig:torus_overall_convg}}
\end{figure}

%
%
%
\subsection{Non-autonomous advection problem\label{sec:non_auto}}
In the final experiments, we consider a setup for the SL scheme that is more common in the literature but does not conform to the convergence and stability theory developed previously.  The goal of this example is to demonstrate that the stability bound of \Cref{ass:A_approximator} on the remapping operator may be overly restrictive in obtaining a convergent SL scheme.  

For the remapping operator $\MLAD_X$ with polynomial precision $k$ in the ambient space $\mathbb{R}^3$, we select the same radius for each ball used to construct the local approximant at the departure point.  We choose the this value so that every ball contains at least $2M$ points, where $M$ is the dimension of the polynomial space.  For quasi-uniform points considered here, this means that the radii of the balls are proportional to $h_X\sqrt{M}$, effectively fixing $\epsilon$ as a constant.  This also means that the asymptotic cost of solving for each local approximant remains fixed as $h_X$ is refined.  This method of choosing the radii is the same as what is proposed in~\cite{HRW_LPR} to construct the operator $\MLS_X$ \cref{eq:Pol_MLS}.
For the time-step $h_t$, we use the common practice of choosing it to satisfy a given CFL number $C_{\rm cfl}$. With this approach, $h_t = C_{\rm cfl} \frac{h_{X}}{u_{\rm max}}$, $u_{\rm max}$ is the maximum magnitude of the velocity field over the simulation time.
To test this new setup for the SL scheme, we use one of the standard advection problems for the sphere: the non-autonomous deformational flow test case from~\cite{NairLauritzen2010}.  The velocity field for this test is usually written with respect to the spherical coordinate basis, but can be easily converted to the Cartesian basis as 
\begin{align}
\vec{a}(\lambda,\theta,t) = 
\begin{bmatrix}
-\sin\lambda & -\cos\lambda\sin\theta \\
\cos\lambda & -\sin\lambda\sin\theta \\
0 & \cos\theta
\end{bmatrix}
\begin{bmatrix}
\frac{10}{T}\cos\lf(\frac{\pi t}{T}\rt)\sin^2\lf(\lambda - \frac{2\pi t}{T}\rt)\sin (2\theta) + \frac{2\pi}{T}\cos\theta \\
\frac{10}{T}\cos\lf(\frac{\pi t}{T}\rt)\sin\lf(2(\lambda - \frac{2\pi t}{5})\rt)\cos\theta
\end{bmatrix},
\end{align}
where $T=5$, $(\lambda,\theta)\in[-\pi,\pi)\times[-\pi/2,\pi/2]$.  This field transports and deforms the initial condition $v(\lambda, \theta, 0)$ until time $t = T/2$, then reverses direction, returning $v$ to its original position and value by $t = T$. \crcom{Thus}, the initial condition serves as the exact solution to compute the errors in the numerical results at $t=T$.  For the initial condition, we use the two Gaussian bells problem from~\cite{NairLauritzen2010} defined in Cartesian coordinates as
\begin{align}
q(x,y,z,0) = 0.95\lf( e^{-5(r_1(x,y,z)))^2} + e^{-5(r_2(x,y,z))^2} \rt),
\end{align}
where $r_i(x,y,z)$ is the Euclidean distance from $(x,y,z)$ to $(x_i,y_i,z_i)$, and the centers of the bells are given as $(x_1,y_1,z_1)  = (\sqrt{3}/2,1/2,0)$ and $(x_1,y_1,z_1)  = (\sqrt{3}/2,-1/2,0)$.  We use the estimate $u_{\max} = 2.91$ and set $C_{\rm cfl} = 4$ for this problem to determine the time-step as discussed above.  For the time-integrator, we use RK4. 
\begin{figure}[htb]
\centering
\begin{tabular}{cc}
\includegraphics[width=0.45\textwidth]{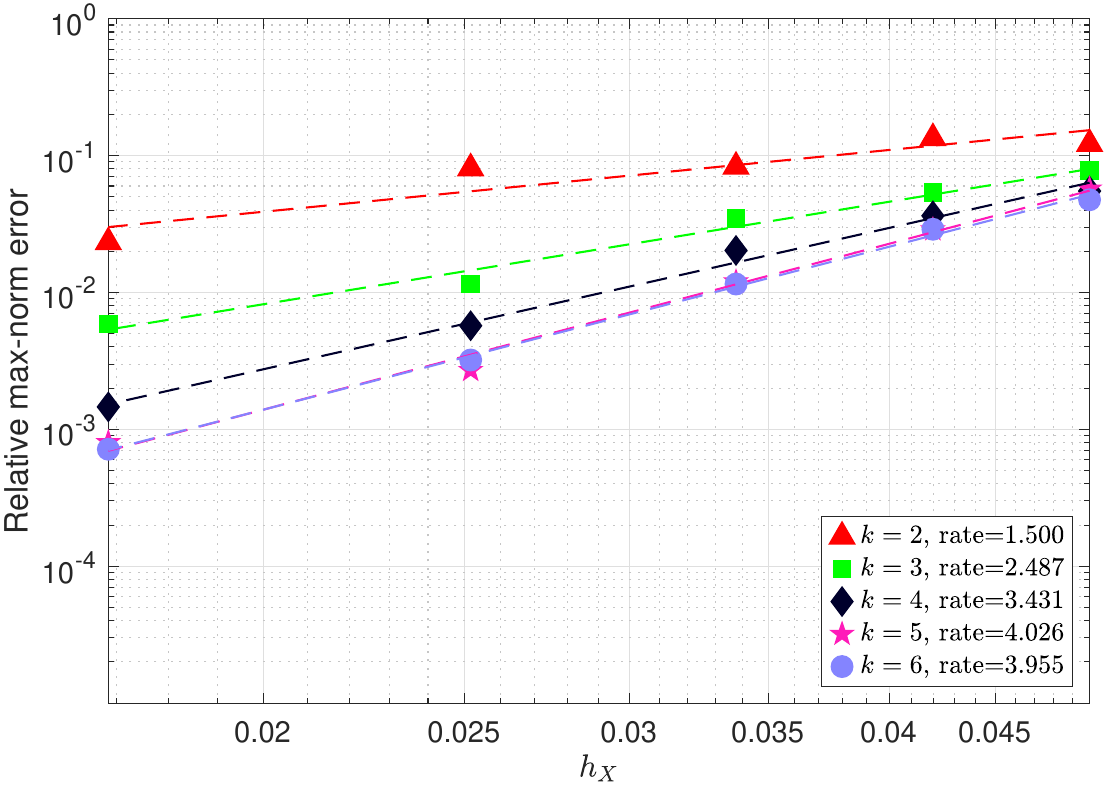} & \includegraphics[width=0.45\textwidth]{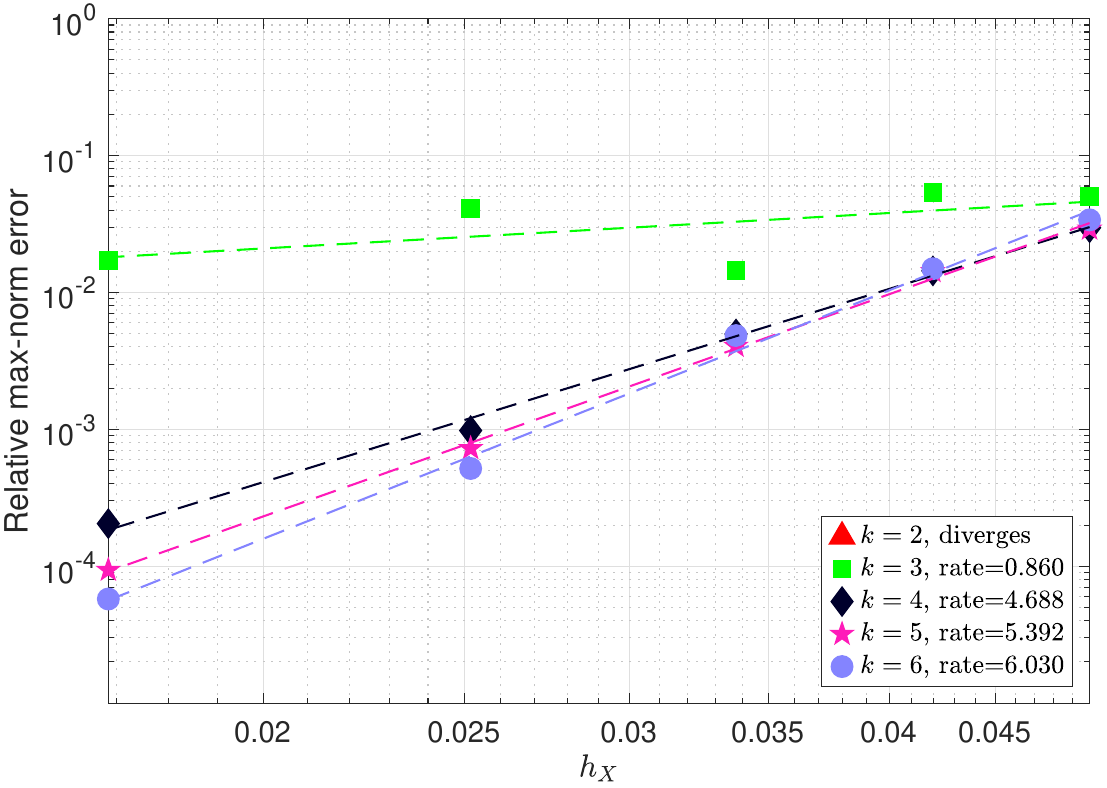}  \\
(a) $\MLAD_X$ & (b) $\MLS_X$
\end{tabular}
\caption{Convergence results for the deformational flow problem on the sphere after one rotation using spatial approximations based on remapping operators (a) $\MLAD_X$ and (b) $\MLS_X$. Time-integrations were done using RK4 with time-step corresponding to and approximate CFL number of 4.  Dashed lines show lines of best fit according to the estimated convergence rates given in the legend.\label{fig:deform_convg}}
\end{figure}
\Cref{fig:deform_convg} (a) shows the convergence of the SL scheme for $\MLAD_X$ with the above setup as $h_X$ is refined.  We see from the figure that the scheme appears to convergence and that the rate of convergence increases with $k$, until $k=6$, where it saturates around 4. The spatial approximation should have convergence of $\mathcal{O}(h_X^{k+1})$ and one would expect to lose an order of convergence in this term from the $1/h_t$ penalty that appears in the SL analysis.  However, it appears that the convergence is reduced by more than this.  The further reduced rate for $k=6$ could be due to the convergence of temporal errors, which we expect to be 4. We also show convergence results for the SL scheme using the $\MLS_X$ remapping operator in \Cref{fig:deform_convg} (b), where the radii of the balls for the local approximations are computed the same way as for $\MLAD_X$.  We see that for $k=2$, the scheme is unstable and diverges as $h_X$ is refined, while for $k=3$, the error is a bit erratic but does appear to slowly converge.  For $k=4,5,6$, the scheme appears to converge, with rates that one might expect according to the discussion above.  In these results, we do not yet observe the limiting fourth order convergence from the time integration. 
%
%
%
\section{Concluding remarks}
We have introduced a new remapping operator $T_X$ in \Cref{sec:ell1_remap} on manifolds that works simply on a point cloud and uses no intermediate geometric approximations, such as, for instance, approximated tangent planes. For theoretical considerations the points have to lie on an algebraic variety which also does not enter the algorithm explicitly. Given this remapping operator, we showed that one can relax the convergence properties of \cite{FalconeFerretti} to such an extent that it now covers the new remapping operator and provides theoretically arbitrary high convergence rates; see \cref{thm:mainerror}.  This is also reflected by the numerical experiments; see, for instance, \Cref{fig:torus_overall_convg}. From the stability theory of~\cite{FalconeFerretti} that has been further developed here, it can be anticipated that the SL schemes would become unstable, especially for larger $k$ using the above setup, since the Lebesgue constants of the remapping operators are larger, as observed in \Cref{sec:lebesgue}.  However, the last numerical example in \Cref{sec:non_auto} shows that this is not necessarily the case.  In fact, there we only observed divergence for $k=2$ in the $\MLS_X$ operator. This indicates that a more complete stability theory for semi-Lagrangian methods is needed.

\bibliographystyle{siam}
\bibliography{literature}

\end{document}